\documentclass[english,reqno]{amsart}
\usepackage{hyperref}
\usepackage[margin=1in]{geometry}
\usepackage{graphicx}
\usepackage{subfig}
\usepackage{amssymb,amsbsy,amsmath,amsfonts,amssymb,amscd}
\usepackage{comment}
\usepackage{verbatim}
\usepackage[dvipsnames]{xcolor}
\usepackage{enumitem}
\usepackage[numbers]{natbib} 
\usepackage{algorithm,algpseudocode}
\usepackage{setspace}

\usepackage{bookmark}
\bookmarksetup{numbered}

\newtheorem{theorem}{Theorem}[section]
\newtheorem{corollary}{Corollary}[section]
\newtheorem{definition}{Definition}[section]
\newtheorem{proposition}{Proposition}[section]

\newtheorem{example}{Example}[section]
\DeclareMathOperator*{\argmin}{argmin}
\newtheorem{remark}{Remark}[section]
\newtheorem{lemma}{Lemma}[section]

\DeclareMathOperator{\sign}{sgn}

\title{Global semiconcavity of solutions to first-order Hamilton--Jacobi equations with state constraints}

\begin{document}

\author{Yuxi Han}

\begin{abstract}
We focus on the global semiconcavity of solutions to first-order Hamilton--Jacobi equations with state constraints, especially for the Hamiltonian $H(x, \beta):=|\beta|^p-f(x)$ with $p \in (1, 2]$. We first show that the solution is locally semiconcave, and the semiconcavity constant at each point depends on the first time a corresponding minimizing curve emanating from this point hits the boundary. Then, with appropriate conditions on $Df$, we prove that for any such minimizing curve, the time it takes to hit the boundary of the domain is $+\infty$, and as a consequence, the solution is globally semiconcave. Moreover, the condition on $Df$ is essentially optimal with examples in one-dimensional space. The proofs employ the Euler-Lagrange equations and techniques in weak KAM theory.      
\end{abstract}

\thanks{The author is supported by NSF CAREER grant DMS-1843320.}

\maketitle

\section{Introduction}\label{sec:intro}
In this paper, we study the semiconcavity of solutions to first-order Hamilton--Jacobi equations with state constraints. Consider the following constrained problem on an open, bounded and connected set $\Omega \subset \mathbb{R}^n$,

\begin{equation}\label{eqn:scp}
    \left\{\begin{aligned}
    u+H(x, Du) & \leq 0 \quad \text{in } \Omega, \\
    u+H(x, Du) & \geq 0 \quad \text{on } \overline{\Omega} ,
    \end{aligned}
    \right.
\end{equation}
where the function $H:\overline{\Omega} \times \mathbb{R}^n \to \mathbb{R}$ is a given continuous Hamiltonian. A function $u: \overline{\Omega} \to \mathbb{R}$ is called a constrained viscosity solution to \eqref{eqn:scp} if $u$ is a subsolution to 
\begin{equation}\label{eqn:scpeq}
    u+H(x, Du)=0
\end{equation}
in $\Omega$ and a supersolution to \eqref{eqn:scpeq} on $\overline{\Omega}$. The existence and the uniqueness of the solution $u \in \mathrm{C} (\overline{\Omega})$ to this constrained problem are guaranteed by appropriate assumptions on $H$ and $\Omega$ (see \cite{Capuzzo-Dolcetta1990,Soner1986}). If $H=H(x, \beta)$ is convex in $\beta$, then the unique constrained solution has the optimal control formula
\begin{equation}\label{eqn:ocf}
    u(x) = \inf \left\lbrace \int_0^\infty e^{-s}\left(L\left(\gamma(s), -\dot{\gamma}(s)\right)\right)ds: \gamma\in \mathrm{AC}([0,\infty);\overline{\Omega}), \gamma(0) = x\right\rbrace,
\end{equation}
where $L:\overline{\Omega} \times \mathbb{R}^n \to \mathbb{R}$ is the Legendre transform of $H$ and $\mathrm{AC}([0,\infty);\overline{\Omega})$ denotes the collection of all the absolutely continuous functions $\gamma : [0, \infty) \to \overline{\Omega}$.
Note that all the admissible paths in the representation formula above are restricted in $\overline{\Omega}$. For the special case where $H(x,\beta):= \left| \beta \right|^p -f(x)$ for $p \in (1, 2]$ and $f\in \mathrm{C}(\overline{\Omega})\cap W^{1,\infty}(\Omega)$, the Legendre transform of $H$ is $L(x, v) =C_p|v|^q+f(x)$ where $C_p = q^{-1} p^{-\frac{q}{p}} $ and $\frac{1}{p} + \frac{1}{q} = 1$. In particular, $q \geq 2$. With appropriate assumptions on $H$, for any $x$, the infimum in \eqref{eqn:ocf} is attained, that is, there exists a minimizing curve for every $x$, as is defined in Definition \ref{def:minimizer}.

We are interested in figuring out fine properties of the solution $u$, especially about whether it is globally semiconcave in the whole domain. Semiconcavity is a property that plays an essential role in optimization (see \cite{cannarsa2004semiconcave}). Roughly speaking,
it can be thought of as the Hessian of the solution $u$ being bounded from above, even though the solution $u$ may not be differentiable everywhere.

 It is not hard to show that the solution $u$ is locally semiconcave. As is proved in \cite[Theorem 16]{YuTu2022}, the local semiconcavity constant of $u$ at a specific point $x$ is bounded by the inverse of its distance to the boundary, i.e., $\mathrm{dist}(x, \partial \Omega)^{-1}$. But this does not imply global semiconcavity since $\mathrm{dist}(x,\partial \Omega)^{-1}$ blows up near the boundary. In this paper, we first improve this local bound by showing that the local semiconcavity constant of $u$ at a specific point $x$ only depends on the inverse of the time a minimizing curve emanating from $x$ takes to hit the boundary of the domain. With this new local bound, to prove the global seminconcavity of $u$, it suffices to find a lower bound of the time for minimizing curves to hit the boundary.

 To the best of our knowledge, global semiconcavity and the first time minimizing curves hit the boundary have not been investigated in the literature and are nontrivial. For instance, in the specific case where $H(x,\beta):= \left| \beta \right|^p -f(x)$ for some function $f\in \mathrm{C}(\overline{\Omega})\cap W^{1,\infty}(\Omega)$, the assumption that $f$ is semiconcave in the domain does not guarantee the solution $u$ is globally semiconcave. See an example in one-dimensional space in Section 4 where $f$ is semiconcave but the semiconcavity constant of $u$ blows up near the boundary of the domain, and in this example, all the minimizing curves take finite time to hit the boundary of the domain.

 It is proved in \cite{YuTu2022} that if $f$ can be extended to a function $\Tilde{f} \in \mathrm{C}(\mathbb{R}^n)$ by setting $\Tilde{f}=0$ in $\Omega^c$ so that $\Tilde{f}$ is semiconcave in $\mathbb{R}^n$, then $u$ is globally semiconcave (see \cite{YuTu2022}). For instance, if $f\in \mathrm{C}_c^2(\Omega)$, then $f$ can be extended to a semiconcave function on $\mathbb{R}^n$ by setting $f=0$ outside $\Omega$. We would like to figure out more relaxed conditions on $f$ that guarantee the global semiconcavity of the solution $u$. In the literature, typically, semiconcavity is proved by doubling variable method or the vanishing viscosity process. However, we believe that the global semiconcavity in our setting cannot be directly deduced via doubling variable method. Instead, we employ the ideas from optimal control and weak KAM theory to link semiconcavity with  properties of minimizing curves, in particular, the first time they hit the boundary (see \cite{evans2001effective, Fathi2014, tran_hamilton-jacobi_2021}). It turns out that suitable conditions on $f$ imply a lower bound for the first time minimizing curves hit the boundary.



\subsection{Relevant literature}
Viscosity solutions with state
constraints is of great interest and there is a lot of work in the literature dedicated to this subject. It starts with \cite{Soner1986} in the settings of optimal control theory, followed by further results like \cite{Capuzzo-Dolcetta1990, Ishii1996, ishii_class_2002}. The asymptotic behavior of solutions to \eqref{eqn:scp} in various settings can be found in \cite{ishii_vanishing_2017,kim_state-constraint_2020,mitake_asymptotic_2008,tu2021vanishing}. For local semiconcavity of Hamilton--Jacobi equations with state constraints, see \cite{CANNARSA2008616, Sinestrari1995SemiconcavityOS}. There are also various results regarding semiconcavity of different types of equations, for instance, \cite{Paolo2010, BRCQ2010, CAROFF2006287, PIGNOTTI2005197, Chen2016, Sinestrari1995SemiconcavityOS, Thomas2010}. Second-order Hamilton--Jacobi equations with state constraints are studied in various work, for instance, \cite{Lasry1989, por2004, alessio_asymptotic_2006}.

However, the question about whether the constrained solution is globally semiconcave has not been thoroughly studied in the literature. Detailed information about the first time a minimizing curve hits the boundary is barely known.

\subsection{Settings}
Before introducing the assumptions used in this paper, we first give the definitions of semiconcavity and local semiconcavity.
\begin{definition}\label{def:semi}
A function $f$ is semiconcave in an open set $\Omega \subset \mathbb{R}^n$ if there exists a constant $C \geq 0$ so that
\begin{equation*}
    f(x+h)+f(x-h)-2f(x) \leq C |h|^2
\end{equation*}
for all $x, h \in \mathbb{R}^n$ such that $[x-h,x+h] \subset \Omega$, where $[x-h,x+h]:= \left\{ x+sh \in \mathbb{R}^n: -1\leq s \leq 1 \right\}$. The constant $C$ above is  called a semiconcavity constant for $f$ in $\Omega$.
\end{definition}  

\begin{definition}
A function $f$ is locally semiconcave in an open set $\Omega \subset \mathbb{R}^n$ if for any $x\in \Omega$, there exists a constant $C_x >0$ which depends on $x$ so that
\begin{equation*}
    f(x+h)+f(x-h)-2f(x) \leq C_x|h|^2
\end{equation*}
for all $h\in \mathbb{R}^n$ such that $[x-h, x+h] \subset \Omega$ with $|h| \leq M_x$ for some constant $M_x>0$ that depends on $x$.
\end{definition}

Next, we introduce a few assumptions that have been used in different literature.

Let $\Omega \subset \mathbb{R}^n$ be an open, bounded, connected domain, and satisfy the condition

(A) there exists a universal pair of positive numbers $(r, h)$ and a function $\alpha \in \mathrm{BUC} (\overline{\Omega}; \mathbb{R}^n)$ such that
$\mathrm{B}(x+t\alpha(x),rt) \subset \Omega, \forall x \in \overline{\Omega}, t \in (0,h]$,

where $\mathrm{BUC}(\overline{\Omega}; \mathbb{R}^n)$ denotes the collection of all the bounded and uniformly continuous functions $g : \overline{\Omega} \to \mathbb{R}^n$.

Assumption (A), which is an assumption on the boundary of $\Omega$, holds true for any bounded domain with $\mathrm{C}^2$ boundary and needs to be assumed to have the comparison principle for \eqref{eqn:scpeq} (see \cite{Soner1986}). In this paper, we always assume $\Omega$ has a $\mathrm{C}^2$ boundary.

Let $H :\overline{\Omega} \times \mathbb{R}^n \to \mathbb{R
}$ be a continuous function. The following are several assumptions on $H$ that have been used in the literature.

(H1) $\displaystyle \lim_{\left|\beta\right| \to \infty}\left(\inf_{x\in \overline{\Omega}}\frac{H\left(x,\beta \right)}{\left|\beta \right|}\right) = +\infty$. 



(H2) The map $\beta \mapsto H(x, \beta)$ is convex, $\forall x \in \overline{\Omega}$.

(H3a) There exists a modulus $\omega_H:[0,+\infty) \to [0, +\infty)$, which is a nondecreasing function such that $w_H(0^+)=0$ and 
\begin{equation*}
    \left\{ \begin{aligned}
      &\left|H\left(x,\beta\right)-H(y,\beta)\right| \leq \omega_H \left(|x-y|(1+|\beta|)\right)\\
      &\left|H(x,\beta)-H\left(x,\tilde{\beta}\right)\right| \leq \omega_H \left(\left|\beta-\tilde{\beta}\right|\right)
    \end{aligned}
    \right.
\end{equation*}
$\forall x, y \in \overline{\Omega}$ and $\beta, \tilde{\beta} \in \mathbb{R}^n$.

(H3b)$\forall R>0$, there exists a modulus $\omega_R:[0,+\infty) \to [0, +\infty)$, which is nondecreasing with $\omega_R(0^+)=0$ and 
\begin{equation*}
    \left\{ \begin{aligned}
      &|H\left(x,\beta\right)-H\left(y,\beta\right)|\leq \omega_R(|x-y|)\\
      &\left|H\left(x,\beta \right)-H\left(x,\tilde{\beta} \right)\right| \leq \omega_R\left(\left|\beta-\tilde{\beta}\right|\right)
    \end{aligned}
    \right.
\end{equation*}
$\forall x, y \in \overline{\Omega}$, $\beta, \tilde{\beta} \in \mathbb{R}^n$ with $\left|\beta\right|, \left|\tilde{\beta}\right| \leq R$.

Actually, we only need $H$ to be coercive, i.e., $\lim_{\left|\beta\right| \to \infty}\left(\inf_{x\in \overline{\Omega}}H\left(x,\beta\right)\right) = +\infty$, instead of assumption (H1). Since we have a priori estimates on the solution $u$ and its gradient $Du$, i.e., $\|u\|_\infty+\|Du\|_\infty \leq C$ for some constant $C$, if $H$ does not satisfy assumption (H1), we can modify $H$ for
$\left|\beta\right|\geq C$ so that assumption (H1) is satisfied.

Assumption (H1) guarantees the existence of the constrained viscosity solution to \eqref{eqn:scp}. Detailed proof of the existence is provided in Appendix for the reader's convenience. The uniqueness of the constrained viscosity solution follows from a general comparison principle, where we need assumption (H3a) or (H3b), together with assumption (A) on the domain. The following is a general comparison principle for \eqref{eqn:scpeq}, which is stated here for completeness (see \cite{Capuzzo-Dolcetta1990, Soner1986}). 

\begin{theorem}
Assume $(\mathrm{A})$. Suppose $v_1 \in \mathrm{BUC}(\overline{\Omega};\mathbb{R})$ is a viscosity subsolution of \eqref{eqn:scpeq} in $\Omega$ and $v_2 \in \mathrm{BUC}(\overline{\Omega};\mathbb{R})$ is a viscosity supersolution of \eqref{eqn:scpeq} on $\overline{\Omega}$. If either
\begin{itemize}
    \item[\rm(i)] $(\mathrm{H}3\mathrm{a})$ holds,
or \item[\rm(ii)] $(\mathrm{H}3\mathrm{b})$ holds and $v_1$ is Lipschitz,
then $v_1(x) \leq v_2(x), \forall x \in \overline{\Omega}$.
\end{itemize}
\end{theorem}

Note that (H3b) is weaker than (H3a). Moreover, if we assume (H1), then any subsolution of \eqref{eqn:scpeq} is Lipschitz. Hence, for the uniqueness of the constrained viscosity solution, we only need to assume (H3b) with (H1).


In this paper, we mainly focus on the case where $H(x, \beta) = |\beta|^p -f(x)$ for $p \in (1, 2]$ and $f \in \mathrm{C}(\overline{\Omega}) \cap W^{1, \infty}(\Omega)$ which is semiconcave in $\Omega$ with $f\equiv \min_{x\in\overline{\Omega}} f(x)$ on $\partial \Omega$. It turns out that in this case, under appropriate assumptions on $Df$, we can prove the global semiconcavity of the solution $u$ (see details in Theorem \ref{thm:main}). Furthermore, the conditions on $Df$ turn out to be essentially optimal, at least for $p=q=2$ (see Example \ref{ex:nex}). 

Finally, as is mentioned before, to prove the global semiconcavity, we need to carefully study the first time minimizing curves hit the boundary, the definition of which is stated below.

\begin{definition}\label{def:minimizer}
Given $x\in \overline{\Omega}$, $\xi \in \mathrm{AC}([0,+\infty); \overline{\Omega})$ is called a minimizing curve emanating from $x$ if
\begin{equation*}
    \xi \in \argmin_{\begin{aligned}\gamma\in \mathrm{AC}&([0,\infty);\overline{\Omega}),\\\gamma(0) &= x\end{aligned}}  \int_0^\infty e^{-s}\left(L\left(\gamma(s), -\dot{\gamma}(s)\right)\right)ds,
\end{equation*}
where $L$ is the Legendre transform of the Hamiltonian $H$ in \eqref{eqn:scp}.
Furthermore, we define the first time $\xi$ hits the boundary by
\begin{equation*}
        T_{x, \xi} : = \inf \left\{ s \in [0, +\infty) : \xi (s) \in \partial \Omega \right\},
\end{equation*}
where we take the convention $\inf{\emptyset} = +\infty$.
\end{definition}

\subsection{Main results and proof strategies}

Now, we present our main results about the global semiconcavity of $u$.
\begin{theorem}\label{thm:main}
Let $\Omega \subset \mathbb{R}^n$ be an open, bounded, connected domain with $C^2$ boundary. Suppose $H(x, \beta) = |\beta|^p -f(x)$ for some $p \in (1, 2]$, $f \in \mathrm{C}^1(\overline{\Omega})$ which is semiconcave in $\Omega$. Assume 
\begin{itemize}
    \item[\rm(1)] $f(x) \equiv \min_{x\in \overline{\Omega}}f(x)$ for $ x \in \partial \Omega$,
    \item[\rm(2)] $f(x) > \min_{x\in \overline{\Omega}}f(x)$ for $x \in \Omega$,
    \item[\rm(3)] there exists a constant $C>0$ such that $\left|Df(x)\right| \leq C \left(f(x)-\min_{x\in \overline{\Omega}}f(x)\right)^\frac{1}{p}$ for all $x \in \Omega$. 
\end{itemize} 
Then, the solution $u$ to \eqref{eqn:scp} is a viscosity solution to 
\begin{equation}
    -D^2u \geq \tilde{C} \mathrm{I}_n \quad \text{ in } \Omega
\end{equation}
for some constant $\tilde{C}$, where $\mathrm{I}_n$ is the $n \times n$ identity matrix.
\end{theorem}

\begin{remark} Two comments for Theorem \ref{thm:main} are as follows.
\begin{itemize}
    \item Condition $(3)$ is almost optimal for the global semiconcavity of the solution $u$, at least for the case $p=q=2$. In one-dimensional space, we show if Condition $(3)$ is not satisfied, the solution is not globally semiconcave. See details in Theorem \ref{thm:1d} in Section \ref{sec:one}. 
    \item In general, if $f \not\equiv \min_{x\in \overline{\Omega}} f(x)$ on $\partial \Omega$, one has to study the dynamics of the minimizing curves and the equations to determine whether $u$ is semiconcave or not. In one-dimensional space, two specific examples are given in Section \ref{sec:one} to illustrate that if the minimum of $f$ is only attained in the interior, then different situations can happen.
\end{itemize}

\end{remark}

In the literature, it has been proved that $u$ is at lease locally semiconcave in $\Omega$. In particular, as is shown in \cite[Theorem 16]{YuTu2022}, the local semiconcavity constant is bounded by $\mathrm{dist}(x, \partial \Omega)^{-1}$. However, we emphasize that the direct application of this local bound fails to prove the global semiconcavity of $u$ because $\mathrm{dist}(x, \partial \Omega)^{-1}$ blows up when $x$ is close to the boundary. Thus, to prove Theorem \ref{thm:main}, two major steps are taken in this paper:
\begin{enumerate}[wide, labelwidth=!, labelindent=0pt] 
    \item[\textbf{Step 1: local semiconcavity of $u$.}] In this step, we prove that there exists a more intuitive bound for the local semiconcavity constant of $u$, as is shown in the following proposition.
\begin{proposition}\label{prp:localsemi}
Let $\Omega \subset \mathbb{R}^n$ be an open, bounded, connected domain with $\mathrm{C}^2$ boundary. Suppose 
\begin{itemize}
    \item[\rm(1)] $H\in \mathrm{C}^2\left(\overline{\Omega} \times \mathbb{R}^n \right)$ satisfies $(\mathrm{H}1)$ and $(\mathrm{H}2)$, or
    \item[\rm(2)] $H(x, \beta) := |\beta|^p -f(x)$ for some $p \in (1, 2]$ and $f \in \mathrm{C}(\overline{\Omega}) \cap W^{1, \infty}(\Omega)$ which is semiconcave in $\Omega$.
\end{itemize}
Let $u$ be the constrained viscosity solution to \eqref{eqn:scp}. Fix $x \in \Omega$ and let $\xi \in \mathrm{AC}([0,+\infty); \overline{\Omega})$ be a minimizing curve emanating from x as is defined in Definition \ref{def:minimizer}. Then, for every $T<T_{x, \xi}$, there exists a constant $M_{x,T} >0$ that depends on $x, T$ and a constant $C$ independent of $x, T$ so that
\begin{equation}\label{ieq:ls}
    u(x+h)-2u(x)+u(x-h) \leq C \left(1+\frac{1}{T}\right) |h|^2
\end{equation}
for all $h \in \mathbb{R}^n$ such that $[x-h, x+h] \subset \Omega$ with $|h| \leq M_{x,T}$.
\end{proposition}

\begin{remark}
Later, we only consider case $(2)$ of $H$ in the above proposition to prove the global semiconcavity of $u$. Case $(1)$ is less restrictive in terms of the form of $H$ and might also have its own applications, hence we include it here as well.  
\end{remark}


The proof of this proposition is provided in Section \ref{sec:local}. According to this proposition, the local semiconcavity constant of $u$ is bounded by the inverse of the time that a minimizing curve takes to hit the boundary. Compared with previous results~\cite{YuTu2022} or Corollary \ref{cor:dis}, where the semiconcavity constant is bounded by the inverse of $\mathrm{dist}(x, \partial \Omega)$, \eqref{ieq:ls} provides a better bound for the semiconcavity constant when $x$ is close to the boundary of the domain. More specifically,  when $\mathrm{dist}(x, \partial \Omega)$ is extremely small, although $\mathrm{dist}(x, \partial \Omega)^{-1}\gg 1$, the minimizing curves for $u(x)$ may have relatively slow speed and take a long time to hit the boundary, which makes $1/T$ remain bounded.

    \item[\textbf{Step 2: lower bound for the hitting time.}] From \eqref{ieq:ls} in Proposition \ref{prp:localsemi}, in order to prove global semiconcavity, it suffices to find a uniform lower bound for the first time minimizing curves hit the boundary. We prove that for any $x \in \Omega$, any minimizing curve $\xi$ takes infinite time to hit the boundary of the domain via the Euler-Lagrange equations, that is, $T_{x,\xi}=+\infty$, where $T_{x,\xi}$ is defined in Definition \ref{def:minimizer}. This result is summarized in the following proposition.

\begin{proposition}\label{prp:Tinfty}
Let $\Omega \subset \mathbb{R}^n$ be an open, bounded, connected domain with $C^2$ boundary. Suppose $H(x, \beta) = |\beta|^p -f(x)$ for some $p \in (1, 2]$, $f \in \mathrm{C}^1(\overline{\Omega})$ which is semiconcave in $\Omega$. Assume 
\begin{itemize}
    \item[\rm(1)] $f(x) \equiv \min_{y\in \overline{\Omega}}f(y)$ for all $ x \in \partial \Omega$,
    \item[\rm(2)] $f(x) > \min_{y\in \overline{\Omega}}f(y)$ for all $x \in \Omega$,
    \item[\rm(3)] there exists a constant $C>0$ such that $|Df(x)| \leq C \left(f(x)-\min_{x\in \overline{\Omega}}f(x)\right)^\frac{1}{p}$ for all $x \in \Omega$. 
\end{itemize} 
Let $u$ be the constrained viscosity solution to \eqref{eqn:scp}. Fix $x \in \Omega$ and let $\xi(\cdot)$ be a minimizing curve for $u(x)$. Then, $T_{x, \xi}=+\infty$, that is, any minimizing curve cannot escape the domain $\Omega$ in finite time. 
\end{proposition}

    The proof of Proposition \ref{prp:Tinfty} is shown in Section \ref{sec:high}. One of the main difficulties is that it is hard to track the dynamics of minimizing curves. For a special $f$, one might solve the Euler-Lagrange equations directly to obtain the exact formula of minimizing curves, as is carried out in Appendix for Example \ref{ex:compute}, but in general, it is extremely hard to derive a useful formula of minimizing curves. To overcome this difficulty, we first track the evolution of $f$ and $u$ along the minimizing curves using the Euler-Lagrange equations and the ideas from optimal control theory and weak KAM theory (see \cite{evans2001effective, Fathi2014, tran_hamilton-jacobi_2021}). Then, we proceed by contradiction to obtain a uniform lower bound for $T_{x, \xi}$. Another technical difficulty comes from the lack of regularity of minimizing curves when $1<p<2$. As the Euler-Lagrange equations are employed, regularity of minimizing curves is essential. It is often assumed that $D_{vv}L>0$ in the study of regularity of minimizing curves for general Lagrangians in the literature (see, e.g., \cite{tran_hamilton-jacobi_2021}). In our setting, for $q>2$, $D_{vv}L$ is not positive definite, hence we need to be extremely careful with the regularity of minimizing curves. Specifically, we prove that for a minimizing curve, $\xi$, it is $C^1$, and if $\dot \xi(t) \neq 0$, then $\xi$ is twice differentiable there.
\end{enumerate}
    
Finally, from Proposition \ref{prp:Tinfty} and Proposition \ref{prp:Tinfty}, Theorem \ref{thm:main} is directly deduced and the proof is provided in Section \ref{sec:high}.

\subsection*{Organization of the paper} In Section \ref{sec:local}, we present the proof of the local semiconcavity of constrained viscosity solutions (proof of Proposition \ref{prp:localsemi}). Besides, several important properties of equation \eqref{eqn:scp} are introduced, which will be used later. The global semiconcavity is proved in Section \ref{sec:high}, before which we prove that every minimizing curve for $x\in \Omega$ takes infinite time to escape the domain (proof of Proposition \ref{prp:Tinfty}). In Section \ref{sec:one}, we show the condition on $Df$ is essentially optimal with an example in one-dimensional space. After that, two more examples are provided to illustrate that the solution $u$ can be semiconcave or not semiconcave respectively if $\min_{x\in \overline{\Omega}} f(x)$ is only attained in the interior.

\section{Preliminaries and local semiconcavity} \label{sec:local}

In this section, we prove local semiconcavity using the optimal control formula for the solution $u$. The existence of minimizing curves is well known (see \cite{tran_hamilton-jacobi_2021}). The proof of minimizing curves being $C^1$ is quite standard and we state the result as a lemma below. The proof is included in Appendix for completeness. Similar proofs can be found in \cite{cannarsa2004semiconcave}. 

\begin{lemma}\label{lem:c1}
Let $\Omega \subset \mathbb{R}^n$ be an open, bounded, connected domain with $C^2$ boundary. Suppose $H(x, \beta) := |\beta|^p -f(x)$ for some $p \in (1, 2]$ and $f \in \mathrm{C}^1(\overline{\Omega})$. Let $u$ be the constrained viscosity solution to \eqref{eqn:scp}. Fix $x \in \Omega$ and let $\xi \in \mathrm{AC}([0,+\infty); \overline{\Omega})$ be a minimizing curve for $u(x)$. Then we have the following results:
\begin{itemize}
    \item For any $t\geq 0$, if $\xi(s) \in \Omega$ for all $s\in [0,t]$, then $\xi \in \mathrm{C^1}\left([0, t]; \mathbb{R}^n \right)$.
    \item $\left|\dot{\xi} (t) \right| \leq C$ for any $t \in [0, \infty)$ and some constant $C>0$ that is independent of $x$ and $\xi$.
\end{itemize}
\end{lemma}


\begin{remark}
Regarding the regularity of minimizing curves, we do not know if $\xi(\cdot)$ is twice differentiable in general. But for the case $p=q=2$, Proposition \ref{prp:ELderiH} shows that $\xi(\cdot)$ is a $C^2$ function.
\end{remark}


\begin{proof}[Proof of Proposition \ref{prp:localsemi}]
Fix $x \in \Omega$ and let $\xi \in \mathrm{AC}([0,+\infty); \overline{\Omega})$ be a minimizing curve for $u(x)$. Hence, $\xi(0)=x$. Then, from the optimal control formula, 
$$u(x)=\int^\infty_0 e^{-s} L \left(\xi(s) ,-\dot{\xi}(s) \right) ds ,$$
where $L$ is the Legendre transform of $H$. Fix $T < T_{x, \xi}$, we have $\xi(s) \in \Omega, \forall s \in [0, T]$. 
Note that by the dynamic programming principle,  
\begin{equation}\label{eqn:uint}
    u(x)=\int_0^{T} e^{-s} L \left(\xi(s) ,-\dot{\xi}(s) \right) ds + e^{-T} u(\xi(T)).
\end{equation}
Define $\Tilde{\xi}:[0, +\infty)\to \mathbb{R}^n$ by
\begin{equation}\label{eqn:tildexi}
    \Tilde{\xi}(s):=\left\{
    \begin{aligned}
    &\xi(s)+\left( 1-\frac{s}{T} \right) h, \quad \text{if }  0\leq s\leq T,\\
    &\xi(s),\qquad \qquad \qquad \quad  \text{if } s \geq T,
    \end{aligned}
    \right.
\end{equation}
and $\hat{\xi}:[0, +\infty)\to \mathbb{R}^n$ by
\begin{equation}\label{eqn:hatxi}
    \hat{\xi}(s):=\left\{
    \begin{aligned}
    &\xi(s)-\left( 1-\frac{s}{T} \right) h, \quad \text{if }  0\leq s\leq T,\\
    &\xi(s),\qquad \qquad \qquad \quad  \text{if } s \geq T.
    \end{aligned}
    \right.
\end{equation}
Choose $h$ small enough so that $\Tilde{\xi}(s), \hat{\xi}(s) \in \Omega, \forall \, s\geq 0$. This can be done because there exists $r>0$ such that $B(\xi(s),r) \subset \Omega, \forall \, 0 \leq s \leq T$.
Note that $\tilde{\xi}(0)=x+h$ and $\hat{\xi}(0)=x-h$. By the optimal control formula for $u(x+h)$, $u(x-h)$ and the cost of $\tilde{\xi}$, $\hat{\xi}$, we have
\begin{equation} \label{ieq:h+}
    u(x+h)\leq \int_0^{T} e^{-s} L\left( \xi(s)+\left(1-\frac{s}{T}\right)h ,-\left(\dot{\xi}(s)-\frac{h}{T} \right) \right)ds + e^{-T} u(\xi(T)),
\end{equation}
and
\begin{equation}\label{ieq:h-}
    u(x-h)\leq \int_0^{T} e^{-s} L\left( \xi(s)-\left(1-\frac{s}{T}\right)h ,-\left(\dot{\xi}(s)+\frac{h}{T} \right) \right)ds + e^{-T} u(\xi(T)).
\end{equation}

Hence, from \eqref{eqn:uint}, \eqref{ieq:h+}, and \eqref{ieq:h-}, for $h$ small enough,
\begin{equation}\label{ieq:difquo}
\begin{aligned}
    u(x+h)+u(x-h)-2u(x) &\leq \int_0^T e^{-s} L\left( \xi(s)+\left(1-\frac{s}{T}\right)h ,-\left(\dot{\xi}(s)-\frac{h}{T} \right) \right) \\&+ e^{-s} L\left( \xi(s)-\left(1-\frac{s}{T}\right)h ,-\left(\dot{\xi}(s)+\frac{h}{T} \right) \right) -2 e^{-s} L \left(\xi(s) ,-\dot{\xi}(s) \right) ds.
\end{aligned}
\end{equation}

\begin{itemize}
    \item[(1)]If $H\in \mathrm{C}^2(\overline{\Omega} \times \mathbb{R}^n)$, then $L \in \mathrm{C}^2(\overline{\Omega} \times \mathbb{R}^n)$. By Taylor's theorem, for any $ y \in \mathbb{R}^n$ and $x \in \overline{\Omega}$,
\begin{equation}
\begin{aligned}
     L\left( x+\left( 1-\frac{s}{T}\right)h, -y+\frac{h}{T}\right)=&L(x, -y) +\frac{\partial L}{\partial x}(x, -y) \cdot \left(1-\frac{s}{T}\right)h+\frac{\partial L}{\partial v}(x, -y) \cdot \frac{h}{T}\\ 
     &+ \left(1-\frac{s}{T}\right)^2 \int^1_0 h^\top \frac{\partial^2 L}{\partial x^2}\left(x+\left(1-\frac{s}{T}\right)th, -y+\frac{h}{T}t\right)h(1-t) dt\\
     &+ \frac{2}{T}\left( 1-\frac{s}{T}\right) \int^1_0   h^\top\frac{\partial^2 L}{\partial x \partial v}\left(x+\left(1-\frac{s}{T}\right)th, -y+\frac{h}{T}t\right)h(1-t)dt\\
     &+ \frac{1}{T^2}\int^1_0  h^\top \frac{\partial^2 L}{\partial v^2}\left(x+\left(1-\frac{s}{T}\right)th, -y+\frac{h}{T}t\right)h(1-t) dt\\
     \leq& L(x, -y)+ \frac{\partial L}{\partial x}(x, -y) \cdot \left(1-\frac{s}{T}\right)h+\frac{\partial L}{\partial v}(x, -y) \cdot \frac{h}{T}\\
     &+ C\left(\left(1-\frac{s}{T}\right)^2 +\frac{1}{T}\left( 1-\frac{s}{T}\right) + \frac{1}{T^2}\right)|h|^2  \left\|D^2L\right\|_{L^\infty \left(\overline{\Omega}\times \mathrm{B}\left(-y,\frac{|h|}{T}\right) \right)}
\end{aligned}
\end{equation}
and similarly
\begin{equation}
    \begin{aligned}
     L\left( x-\left(1-\frac{s}{T}\right)h ,-y-\frac{h}{T} \right) 
\leq& L(x, -y)- \frac{\partial L}{\partial x}(x, -y) \cdot \left(1-\frac{s}{T}\right)h-\frac{\partial L}{\partial v}(x, -y) \cdot \frac{h}{T}\\
     &+ C\left(\left(1-\frac{s}{T}\right)^2 +\frac{1}{T}\left( 1-\frac{s}{T}\right) + \frac{1}{T^2}\right)|h|^2  \left\|D^2L\right\|_{L^\infty \left(\overline{\Omega}\times \mathrm{B}\left(-y,\frac{|h|}{T}\right) \right)}
    \end{aligned}
\end{equation}

which implies
\begin{equation}\label{ieq:dql}
\begin{aligned}
 &L\left( \xi(s)+\left(1-\frac{s}{T}\right)h ,-\left(\dot{\xi}(s)-\frac{h}{T} \right) \right) + L\left( \xi(s)-\left(1-\frac{s}{T}\right)h ,-\left(\dot{\xi}(s)+\frac{h}{T} \right) \right) -2L \left(\xi(s) ,-\dot{\xi}(s) \right)\\
 \leq &C\left(\left(1-\frac{s}{T}\right)^2 +\frac{1}{T}\left( 1-\frac{s}{T}\right) + \frac{1}{T^2}\right)|h|^2  \left\|D^2L\right\|_{L^\infty \left(\overline{\Omega}\times \mathrm{B}\left(-\dot{\xi}(s), 1\right) \right)}\\
 \leq & C\left(\left(1-\frac{s}{T}\right)^2 +\frac{1}{T}\left( 1-\frac{s}{T}\right) + \frac{1}{T^2}\right)|h|^2,     
\end{aligned}
\end{equation}
where we use $\left\|\dot{\xi}\right\|_\infty < C$ from Lemma \ref{lem:c1} and $h$ is chosen to be small enough so that
\begin{equation}\label{ieq:hsmall}
\left|\frac{h}{T}\right| \leq 1.    
\end{equation}

Plugging \eqref{ieq:dql} into \eqref{ieq:difquo}, we get
\begin{equation}\label{ieq:semicT}
\begin{aligned}
     u(x+h)+u(x-h)-2u(x) &\leq C |h|^2 \int_0^Te^{-s} \left( 1-\frac{s}{T}\right)^2 ds  + C |h|^2 \int_0^T \frac{e^{-s}}{T}ds+C|h|^2 \int_0^T \frac{e^{-s}}{T^2}ds\\
     & \leq C|h|^2 \int_0^T e^{-s}ds+ C\frac{|h|^2}{T}+C\frac{|h|^2}{T} \int_0^1e^{-sT}ds\\
     &\leq C\left(1+\frac{1}{T}\right)|h|^2,\\
\end{aligned}
\end{equation}
where $C=C\left(\left\|\dot{\xi}\right\|_\infty\right)$. Therefore, by Lemma \ref{lem:c1}, the constant $C$ is independent of $x, T$.
    \item[(2)] If $H(x, \beta):=|\beta|^p-f(x)$, then $L(x, v)= C_p|v|^q+f(x)$, where $C_p = q^{-1} p^{-\frac{q}{p}} $ and $\frac{1}{p} + \frac{1}{q} = 1$. Inequality \eqref{ieq:difquo} becomes
    \begin{equation}
    \begin{aligned}
    & u(x+h)+u(x-h)-2u(x) \\
    \leq &\int_0^T e^{-s} C_p \left( \left\lvert  \dot{\xi}(s)-\frac{h}{T}\right\rvert  ^q + \left\lvert  \dot{\xi}(s)+\frac{h}{T} \right\rvert  ^q -2\left\lvert  \dot{\xi}(s) \right\rvert  ^q\right) ds\\
    &+\int_0^Te^{-s}\left(f \left(\xi(s)+\left(1-\frac{s}{T}\right)h\right) + f \left(\xi(s)-\left(1-\frac{s}{T} \right)h\right) -2f\left(\xi(s)\right) \right)ds\\
    \leq &\int_0^T e^{-s} C_p \left( \left\lvert  \dot{\xi}(s)-\frac{h}{T}\right\rvert  ^q + \left\lvert  \dot{\xi}(s)+\frac{h}{T} \right\rvert  ^q -2 \left\lvert  \dot{\xi}(s)\right\rvert  ^q\right) ds\\
    &+C \left\lvert  h\right\rvert  ^2 \int_0^Te^{-s} \left( 1-\frac{s}{T}\right)^2 ds,
    \end{aligned}
    \end{equation}
where the second inequality follows from the semiconcavity of $f$. By Taylor's expansion of $|\cdot|^q$ and similar arguments as in part $(2)$, we get
\begin{equation*}
    u(x+h)+u(x-h)-2u(x)\leq C\left(1+\frac{1}{T}\right)|h|^2
\end{equation*}
for some constant $C$ independent of $x, T$ and for all $|h|$ small enough. (Also see \cite{YuTu2022} for details)
\end{itemize}

\end{proof}

Note that \eqref{ieq:ls} only holds when $|h|$ is small enough and the smallness of $|h|$ depends on $T$ from \eqref{ieq:hsmall}, and also from \eqref{eqn:tildexi}, \eqref{eqn:hatxi} where we make $\Tilde{\xi}(s), \hat{\xi}(s) \in \Omega, \forall \, s\geq 0$ by choosing $|h|$ small. 

Since the speed of minimizing curves is uniformly bounded, we can then deduce that the semiconcavity constant of $u$ at the point $x$ depends on $\mathrm{dist}(x, \partial \Omega)$, which we summarize below in Corollary \ref{cor:dis}. (See also \cite{YuTu2022}) 



\begin{corollary}\label{cor:dis}
Under the conditions of Proposition \ref{prp:localsemi}, there exists a constant $C > 0$ independent of $x \in \Omega$ so that $\forall x \in \Omega$, 
\begin{equation}
    u(x+h)-2u(x)+u(x-h)\leq \frac{C}{\mathrm{dist}(x, \partial \Omega)}|h|^2,
\end{equation}
for any $h\in\mathbb{R}^n$ such that $[x-h, x+h] \subset \Omega$ with $|h| \leq M_x$, where $M_x$ is constant that only depends on $x$.
\end{corollary}

\begin{proof}
Fix $x \in \Omega$ and let $\xi \in \mathrm{AC}([0,+\infty); \overline{\Omega})$ be a minimizing curve for $u(x)$. Since $\left\|\dot{\xi}\right\|_\infty \leq C$ where $C$ is independent of $x$ and $\xi$, in Proposition \ref{prp:localsemi}, we can choose $\displaystyle T = \frac{\mathrm{dist}(x, \partial \Omega)}{C+1}$. Indeed, $\forall \, t \leq T$, 
\begin{equation*}
    \left|\xi(t)-\xi(0)\right|\leq \int_0^t\left|\dot{\xi}(s)\right|ds \leq tC \leq TC = \frac{C}{C+1} \mathrm{dist}\left(x, \partial \Omega\right) < \mathrm{dist}(x, \partial \Omega).
\end{equation*}
Hence, $\xi(t) \in \Omega$, $\forall \, t \in [0, T] $. Then from \eqref{ieq:semicT}, we have 
\begin{equation}
\begin{aligned}
     u(x+h)+u(x-h)-2u(x) & \leq C\left(1 + \frac{1}{\mathrm{dist}\left(x, \partial \Omega\right)}\right)|h|^2\\
     &\leq \frac{C}{\mathrm{dist}\left(x, \partial \Omega \right)}|h|^2,
\end{aligned}
\end{equation}
where the constant $C$ is independent of $x$.
\end{proof}



Next, we introduce some properties of equation \eqref{eqn:scpeq} and the solution $u$. The following lemma illustrates the relation between the function $f$ and the solution $u$ that they vanish at the same points, which follows from the optimal control formula of the constrained solution. See \cite{YuTu2022} for proof. 

\begin{lemma}\label{lem:f=0} Let $\Omega \subset \mathbb{R}^n$ be an open, bounded, connected domain with $\mathrm{C}^2$ boundary. Let $H(x, \beta) : = |\beta|^p -f(x)$ for some $p \in (1, 2]$ and $f \in \mathrm{C}(\overline{\Omega}) \cap W^{1, \infty}(\Omega)$ and $u$ be the constrained viscosity solution to \eqref{eqn:scp}. Assume $f\geq 0$ in $\Omega$. Then $u(x) = 0$ if and only if $f(x) = 0$. In particular, $f \equiv 0$ implies $u \equiv 0$.
\end{lemma}

Regarding supersolutions and subsolutions to \eqref{eqn:scpeq}, it is straightforward to check that $f$ itself is a supersolution to \eqref{eqn:scpeq} on $\overline{\Omega}$. It turns out that it is crucial to have a subsolution to \eqref{eqn:scpeq} in $\Omega$ for the existence of a minimizing curve $\xi$ for every $x\in \Omega$ such that $T_{x, \xi} = +\infty$. If $f \equiv \min_{x\in \overline{\Omega}} f(x) = 0$ on $\partial \Omega$, the condition $|Df| \leq C f^\frac{1}{p}$ for some constant $C$ guarantees that there exists a small constant $c_0>0$ such that $c_0f$ is a subsolution to \eqref{eqn:scpeq} in $\Omega$. 

\begin{lemma}\label{lem:supsub}
Let $\Omega \subset \mathbb{R}^n$ be an open, bounded, connected domain with $\mathrm{C}^2$ boundary. Let $H(x, \beta) := |\beta|^p -f(x)$ for some $p \in (1, 2]$ and $f \in \mathrm{C}(\overline{\Omega}) \cap W^{1, \infty}(\Omega)$. Then 
\begin{itemize}
    \item[\rm(1)] $f$ is a supersolution to \eqref{eqn:scpeq} on $\overline{\Omega}$.
    \item[\rm(2)]Assume $f=0$ on $\partial \Omega$, $f>0$ in $\Omega$ and there exists a constant $C>0$ such that $|Df| \leq C f^\frac{1}{p}$ in $\Omega$. Then there exists a constant $ c_0 \in \left(0, \frac{1}{2} \right]$ such that $c_0 f$ is subsolution to \eqref{eqn:scpeq} in $\Omega$. 
\end{itemize}
\end{lemma}
\begin{proof}
\begin{itemize}
    \item[(1)] Let $\phi \in \mathrm{C}^\infty(\overline{\Omega})$ such that $\phi - f$ attains its maximum at some point $x_0 \in \overline{\Omega}$. Then
\begin{equation*}
    f(x_0)+\left|D\phi(x_0)\right|^p -f(x_0) \geq 0.
\end{equation*}
Hence, $f$ is a supersolution to \eqref{eqn:scpeq}.
    \item[(2)]Choose $\displaystyle c_0:=\min\left\{\frac{1}{2}, 2^{-1/p}\frac{1}{C}\right\}$. Let $x \in \Omega$ and $v \in D^+f(x)$. Then $c_0 v \in D^+\{c_0 f(x)\}$. 
    \begin{equation*}
        \begin{aligned}
                 c_0f+c_0^p|v|^p-f=f\left(c_0+\frac{c_0^p|v|^p}{f} -1\right) \leq f\left(\frac{1}{2}+\frac{|v|^p}{2C^pf} -1\right) \leq 0.
        \end{aligned}
    \end{equation*}
    Therefore, $c_0f$ is a subsolution to \eqref{eqn:scpeq} in $\Omega$.
\end{itemize}
\end{proof}


\section{Global semiconcavity} \label{sec:high}

In this section, we prove the main result of global semiconcavity of the solution $u$, i.e., Theorem \ref{thm:main}. As is explained in Section \ref{sec:intro}, two steps are taken in the proof of Theorem \ref{thm:main}. Step 1 of local semiconcavity has been shown in Section \ref{sec:local}. We focus on step 2 of lower bound for the hitting time of minimizing curves in this section, namely Proposition \ref{prp:Tinfty}. 

We first derive the Euler-Lagrange equations and Hamilton's ODE for minimizing curves and also prove the fact that the constrained viscosity solution $u$ is differentiable along minimizing curves. The proof of the differentiability property of $u$ along minimizing curves relies on the construction of two smooth functions touching $u$ from both above and below and is related to weak KAM theory. We show these two smooth functions share the same gradient at the point where they touch $u$. The results are presented in the following Proposition. 


\begin{proposition}\label{prp:ELderiH}
Let $\Omega \subset \mathbb{R}^n$ be an open, bounded, connected domain with $\mathrm{C}^2$ boundary. Suppose $H(x, \beta) := |\beta|^p -f(x)$ for some $p \in (1, 2]$ and $f \in \mathrm{C}^1(\overline{\Omega})$. Let $u$ be the constrained viscosity solution to \eqref{eqn:scp}. Fix $x \in \Omega$ and let $\xi \in \mathrm{AC}([0,+\infty); \overline{\Omega})$ be a minimizing curve for $u(x)$.
\begin{itemize}
    \item[\rm(1)] Then, for $t \in \left[0, T_{x, \xi}\right)$, $\left| \dot{\xi}(\cdot)\right|^{q-2} \dot{\xi}(\cdot)$ is a $\mathrm{C}^1$ function and $\xi(\cdot)$ satisfies the Euler-Lagrange equations
\begin{equation}\label{eqn:ELp<2}
\frac{d}{dt}\left(e^{-t}C_p q \left| \dot{\xi}(t)\right|^{q-2} \dot{\xi}(t) \right)= e^{-t} Df\left(\xi(t)\right),
\end{equation}
where 
    $C_p = q^{-1} p^{-\frac{q}{p}}$ and $\frac{1}{p} + \frac{1}{q} = 1.$ In particular, if $p=2$, $\xi \in \mathrm{C}^2\left[0, T_{x, \xi}\right)$.
   \item[\rm(2)] The solution $u$ is differentiable at $\xi(t)$ for any $t \in (0, T_{x, \xi})$ and
\begin{equation}\label{eqn:uderi}
        Du\left( \xi(t) \right)=-C_pq\left|\dot{\xi}(t)\right|^{q-2} \dot{\xi}(t).
\end{equation}
\item[\rm(3)]Define $\eta(s):= D_vL\left(s, 
\xi(s), \dot{\xi}(s)\right)$ where $L(s,x,v):=e^{-s}\left(C_p|v|^{q} +f(x)\right)$ with $\displaystyle C_p = q^{-1} p^{-\frac{q}{p}}, \frac{1}{p} + \frac{1}{q} = 1$.
Then, for any $t \in \left(0, T_{x, \xi}\right)$, the pair $\left(\xi, \eta \right)$ solves Hamilton's ODE
\begin{equation}
    \left\{ \begin{aligned}
    \dot{\xi}(s)&=e^\frac{s}{q-1} p\left| \eta(s) \right|^{p-2}\eta(s) \\
    \dot{\eta}(s)&=e^{-s}Df\left( \xi(s) \right)
    \end{aligned}
    \right.
\end{equation}
for $s \in (0,t)$, with the conditions
\begin{equation*}
    \left\{\begin{aligned}
    \xi(0)&=x\\
    \eta(t)&= -e^{-t}Du\left(\xi(t)\right).
    \end{aligned}
    \right.
\end{equation*}
\end{itemize}
\end{proposition}

\begin{proof}
\begin{itemize}
    \item[(1)] The optimal control formula for the solution $u$ is
\begin{equation}\label{eqn:opf}
    u(x) = \inf \left\lbrace \int_0^\infty e^{-s}\left(C_p|\dot{\gamma}(s)|^{q} +f(\gamma(s))\right)ds: \gamma\in \mathrm{AC}([0,\infty);\overline{\Omega}), \gamma(0) = x\right\rbrace,
\end{equation}
where 
\begin{equation*}
    C_p = q^{-1} p^{-\frac{q}{p}} \qquad\text{and}\qquad \frac{1}{p} + \frac{1}{q} = 1.
\end{equation*}

Fix $b >0$ such that $b<T_{x, \xi}$, and define $$\displaystyle I_b[\gamma] := \int_0^b e^{-s} \Big( C_p|\dot{\gamma}(s)|^{q} +f(\gamma(s)) \Big)ds.$$
Since $\xi$ is a minimizer of \eqref{eqn:opf}, $\xi$ is also a minimizer of the problem
\begin{equation}
    \min \left\{ I_b [\gamma] ; \gamma \in \mathrm{AC}\left([0, b]; \Omega \right) , \gamma(0) = x, \gamma(b) =\xi(b) \right\}.
\end{equation}
Let $\zeta \in \mathrm{C}^\infty\left([0, b]; \mathbb{R}^n \right)$ with $ \zeta(0) = 0, \zeta(b) =0$, and $\tau \in \mathbb{R}$. Hence, $\xi(s) +\tau \zeta(s) \in \Omega$, $ \forall s \in [0,b]$, for $|\tau|$ small enough. By calculus of variation, we know
\begin{equation*}
\begin{aligned}
    0&=\left.\frac{d}{d\tau}\left( I[\xi +\tau \zeta]\right) \right|_{\tau=0}\\
    &=\int_0^b e^{-s} \left( C_p q \left| \dot{\xi}(s)\right|^{q-2} \dot{\xi} (s) \cdot \dot{\zeta}(s) +Df\left( \xi (s)\right) \cdot \zeta (s) \right) ds,\\
\end{aligned}
\end{equation*}
which is true for any $\zeta \in \mathrm{C}^\infty\left([0, b]; \mathbb{R}^n \right)$ with $ \zeta(0) = 0, \zeta(b) =0$.

Since $\xi \in \mathrm{C}^1[0,b]$ and $f \in \mathrm{C}^1\left(\overline{\Omega}\right)$, for any $t \in [0, b]$,
\begin{equation}
    e^{-t} C_p q \left| \dot{\xi}(t)\right|^{q-2} \dot{\xi}(t) = \int_0^te^{-s} Df\left(\xi(s)\right)ds+\tilde{c}
\end{equation}
for some constant $\tilde{c}$, and hence $\left| \dot{\xi}(t)\right|^{q-2} \dot{\xi}(t) \in \mathrm{C}^1[0, b]$. Hence,
\begin{equation}
     \frac{d}{dt}\left(e^{-t}C_p q \left| \dot{\xi}(t)\right|^{q-2} \dot{\xi}(t) \right)= e^{-t} Df\left(\xi(t)\right).
\end{equation}



\item[(2)] Let $x_0 \in \Omega$ and $\xi(\cdot)$ be a minimizing curve for $u(x_0)$. Fix $t \in \left(0, T_{x, \xi}\right)$ and let $y:=\xi(t)$. Then, there exists a constant $b>0$ such that $t+b<T_{x, \xi}$. Since $\left\{\xi(s): s\in [t, t+b]\right\}$ is a compact set, there exists a constant $r>0$ such that $B(\xi(s),r) \subset \Omega$ for any $s\in [t, t+b]$. Therefore, for $x \in \Omega$ close enough to $y$, the curve defined by
\begin{equation*}
    \gamma(s):=
    \xi(s+t)+\frac{b-s}{b}(x-y), \text{ for } s \in [0, b]
\end{equation*}
stays in $\Omega$. Note that $\gamma(0)=x$ and $\gamma(b)=\xi(t+b)$. By the dynamic programming principle, we know
\begin{equation*}
    u(x)= \inf\left\{ \int_0^{b} e^{-s} \left(C_p\left|\dot{\gamma}(s)\right|^q + f\left(\gamma(s)\right) \right) ds + e^{-b} u(\gamma(b)); \gamma \in \mathrm{AC}([0, b]; \overline{\Omega}), \gamma(0)=x \right\}.
\end{equation*}
Hence,
\begin{equation*}
\begin{aligned}
      u(x) &\leq \int_0^{b} e^{-s} \left(C_p\left|\dot{\xi}(s+t)-\frac{x-y}{b}\right|^q + f\left(\xi(s+t)+\frac{b-s}{b}(x-y)\right) \right) ds + e^{-b} u(\xi(b+t))\\
      &=\int_t^{t+b}e^{-s+t}\left(C_p \left|\dot{\xi}(s)-\frac{x-y}{b}\right|^q+f\left(\xi(s)+\frac{b+t-s}{b}(x-y)\right)\right)ds+e^{-b} u(\xi(b+t)),
\end{aligned}
\end{equation*}
which implies
\begin{equation}\label{ieq:eu}
\begin{aligned}
 e^{-t}u(x) \leq & \int_t^{t+b}e^{-s}\left(C_p \left|\dot{\xi}(s)-\frac{x-y}{b}\right|^q+f\left(\xi(s)+\frac{b+t-s}{b}(x-y)\right)\right)ds\\& +e^{-(b+t)} u(\xi(b+t)).
\end{aligned}
\end{equation}
It is also known from the optimal control formula for $u(y)$ that
\begin{equation}\label{eqn:eu}
    e^{-t}u(y) =e^{-t}u\left(\xi(t)\right)= \int_t^{t+b}e^{-s}\left(C_p \left|\dot{\xi}(s)\right|^q+f\left(\xi(s)\right)\right)ds+e^{-(b+t)} u(\xi(b+t)).
\end{equation}
Define
\begin{equation*}
    \phi(x) := \int_t^{t+b}e^{-s}\left(C_p \left|\dot{\xi}(s)-\frac{x-y}{b}\right|^q+f\left(\xi(s)+\frac{b+t-s}{b}(x-y)\right)\right)ds+e^{-(b+t)} u(\xi(b+t)).
\end{equation*}
From \eqref{ieq:eu} and \eqref{eqn:eu}, we have $\phi(x) \geq e^{-t}u(x)$ for any $x$ near $y:=\xi(t)$ and $\phi(y)=e^{-t}u(y)$. Now, compute
\begin{equation*}
\begin{aligned}
      D\phi(y) =&\int_t^{t+b} -\frac{1}{b} e^{-s}  C_p q\left|\dot{\xi}(s)\right|^{q-2} \dot{\xi}(s) +e^{-s}Df\left(\xi(s)\right)\frac{b+t-s}{b}ds\\
      =&\int_t^{t+b} -\frac{1}{b} e^{-s}  C_p q\left|\dot{\xi}(s)\right|^{q-2} \dot{\xi}(s) + \left(e^{-s}C_pq\left| \dot{\xi}(s)\right|^{q-2} \dot{\xi}(s)\right)^\prime\frac{b+t-s}{b}ds\\
      =&\int_t^{t+b} \frac{d}{ds} \left( e^{-s}C_pq\left| \dot{\xi}(s)\right|^{q-2} \dot{\xi}(s) \frac{b+t-s}{b}\right)ds\\
      =& -e^{-t}C_pq\left| \dot{\xi}(t)\right|^{q-2} \dot{\xi}(t) \in e^{-t} D^+u(y).
\end{aligned}
\end{equation*}
Therefore, $-C_pq\left| \dot{\xi}(t)\right|^{q-2} \dot{\xi}(t) \in  D^+u(y)$.
On the other hand, consider
\begin{equation*}
    \gamma(s):=\begin{aligned}
      \xi(s)+\frac{s}{t}(x-y)
    \end{aligned}, \text{ for } s \in [0, t].
\end{equation*}
Again, if $x$ is close enough to $y$, then $\gamma(s) \in \Omega$, $\forall \, s \in [0, t]$. Note that $\gamma(0) = \xi(0)$ and $\gamma(t)=x$. We know
\begin{equation}\label{eqn:eu0}
\begin{aligned}
    u\left(\xi(0)\right) &=  \int_0^{t} e^{-s} \left(C_p\left|\dot{\xi}(s)\right|^q + f\left(\xi(s)\right) \right) ds + e^{-t} u(\xi(t))\\
    &=\int_0^{t} e^{-s} \left(C_p\left|\dot{\xi}(s)\right|^q + f\left(\xi(s)\right) \right) ds +e^{-t} u(y)
\end{aligned}
\end{equation}
 and
 \begin{equation}\label{ieq:eu0}
   u\left(\xi(0)\right) \leq \int_0^{t} e^{-s} C_p\left|\dot{\xi}(s) + \frac{x-y}{t}\right|^q +e^{-s} f\left(\xi(s)+\frac{s}{t}(x-y)\right)ds+e^{-t}u\left(x\right).
 \end{equation}
Define
\begin{equation*}
    \psi(x):=  u\left(\xi(0)\right) -\int_0^{t} e^{-s} \left( C_p\left|\dot{\xi}(s) + \frac{x-y}{t}\right|^q + f\left(\xi(s)+\frac{s}{t}(x-y)\right)\right)ds.
\end{equation*}
Then, from \eqref{eqn:eu0} and \eqref{ieq:eu0}, $\psi(x) \leq e^{-t}u(x)$ for any $x$ near $y:=\xi(t)$ and $\psi(y)=e^{-t}u(y)$. Compute
\begin{equation*}
    \begin{aligned}
      D\psi(y)&=\int_0^t -\frac{1}{t}e^{-s}C_pq\left|\dot{\xi}(s)\right|^{q-2}\dot{\xi}(s)-e^{-s} Df \left(\xi(s) \right) \frac{s}{t}ds\\
      &=\int_0^t -\frac{1}{t}e^{-s}C_pq\left|\dot{\xi}(s)\right|^{q-2}\dot{\xi}(s)-\left(e^{-s}C_pq\left| \dot{\xi}(s)\right|^{q-2} \dot{\xi}(s)\right)^\prime \frac{s}{t}ds\\
      &=-\int_0^t \frac{d}{ds}\left(  e^{-s} C_pq\left| \dot{\xi}(s)\right|^{q-2} \dot{\xi}(s) \frac{s}{t}\right)ds\\
      &=-e^{-t}C_pq\left| \dot{\xi}(t)\right|^{q-2} \dot{\xi}(t) \in e^{-t}D^-u(y).
    \end{aligned}
\end{equation*}
Hence, $-C_pq\left| \dot{\xi}(t)\right|^{q-2} \dot{\xi}(t) \in  D^-u(y)$. Since  $-C_pq\left| \dot{\xi}(t)\right|^{q-2} \dot{\xi}(t) \in  D^+u(y)$ as well, $u$ is differentiable at $y:=\xi(t)$ and 
\begin{equation*}
    Du(y)=Du(\xi(t))= -C_pq\left| \dot{\xi}(t)\right|^{q-2} \dot{\xi}(t).
\end{equation*}
See similar results in \cite{evans2001effective, Fathi2014, tran_hamilton-jacobi_2021}.

\item[(3)] Let $L(s,x,v):= e^{-s}\left( C_p |v|^q +f(x) \right)$. Compute the Legendre transform of $L$ and get
\begin{equation*}
    H(s,x,\beta)= \sup_v \left( \beta \cdot v - L(s,x,v) \right)=e^{\frac{s}{q-1}}|\beta|^p-e^{-s}f(x).
\end{equation*} Define $\eta(s):= D_vL\left(s, 
\xi(s), \dot{\xi}(s)\right)=e^{-s}C_p q\left|\dot{\xi}(s)\right|^{q-2} \dot{\xi}(s)$. From part $(1)$, we know $\eta(s)$ is continuously differentiable for $s \in \left[0, T_{x, \xi}\right)$. 
Since $v \mapsto L(s, x, v)$ is strictly convex, the pair $(\xi, \eta)$ solves Hamilton's ODE
\begin{equation*}
    \left\{\begin{aligned}
     \dot{\xi}(s)&= D_pH(s, \xi(s), \eta(s))=e^\frac{s}{q-1} p \left| \eta (s)\right|^{p-2} \eta(s)\\
     \dot{\eta}(s) &=-D_xH(s, \xi(s), \eta(s))=e^{-s}Df\left(\xi(s)\right).
    \end{aligned}
    \right.
\end{equation*}
Also from part $(2)$, we see
\begin{equation*}
    \eta(t)= e^{-t}C_p q\left|\dot{\xi}(t)\right|^{q-2} \dot{\xi}(t) =-e^{-t}Du\left(\xi(t)\right)
\end{equation*}
for any $t \in \left(0, T_{x, \xi}\right)$.
\end{itemize}
\end{proof}

A direct result of the above proposition is the regularity of minimizing curves, especially for the case $1<p<2$. If the velocity of a minimizing curve is nonzero at some time point, then $\xi$ is twice continuously differentiable there.
\begin{corollary}\label{cor:xic2}
Under the same assumptions in Proposition \ref{prp:ELderiH}, let $x \in \Omega$, $\xi \in \mathrm{AC}([0,+\infty); \overline{\Omega})$ be a minimizing curve for $u(x)$. Then, $\xi(\cdot)$ is twice continuously differentiable at any time $t \in \left[0, T_{x, \xi}\right)$ where $\dot{\xi}(t) \neq 0$. In particular, if $\dot{\xi}(t) \neq 0$ for all $t\in (0, T_{x, \xi})$, then $\xi(\cdot) \in \mathrm{C}^2 \left(0, T_{x, \xi}\right)$. 
\end{corollary}

\begin{proof}
Observe from \eqref{eqn:ELp<2} that
\begin{equation}
\begin{aligned}
      \left| \dot{\xi}(t)\right|^{q-1} &= \frac{e^t}{C_pq} \left|\int_0^te^{-s} Df\left(\xi(s)\right)ds+\tilde{c} \right|\\
      \left| \dot{\xi}(t)\right|&=\left(\frac{e^t}{C_pq}\right)^\frac{1}{q-1}  \left|\int_0^te^{-s} Df\left(\xi(s)\right)ds+\tilde{c} \right|^\frac{1}{q-1},
\end{aligned}
\end{equation}
where $\tilde{c}$ is some constant. Hence, $\left| \dot{\xi}(\cdot)\right|$ is continuously differentiable at any time point $t$ where $\dot{\xi}(t)\neq 0$. Moreover, if $\dot{\xi}(t) \neq 0$,

\begin{equation}
    \dot{\xi}(t)=\frac{e^t}{C_pq\left|\dot{\xi}(t)\right|^{q-2}} \left(\int_0^te^{-s} Df\left(\xi(s)\right)ds+\tilde{c}\right),
\end{equation}
again from \eqref{eqn:ELp<2}. Therefore, $\xi(\cdot)$ is twice continuously differentiable at any time point t where $\dot{\xi}(t) \neq 0$.
\end{proof}

As we can see from Proposition \ref{prp:localsemi}, in order to prove the global semiconcavity of $u$ in the whole domain, it is essential to obtain a lower bound of the time minimizing curves take to hit the boundary of the domain. First, we show that there exists a minimizing curve $\xi$ for every $x \in \Omega$ such that $T_{x, \xi}= +\infty$. Consequently, $u$ is globally semiconcave by Proposition \ref{prp:localsemi}. Moreover, we prove that for any minimizing curve $\xi$ emanating from any $x \in \Omega$, $T_{x, \xi} = +\infty$. 

\begin{proposition}\label{prp:lbtime}
Let $\Omega \subset \mathbb{R}^n$ be an open, bounded, connected domain with $C^2$ boundary. Suppose $H(x, \beta) := |\beta|^p -f(x)$ for some $p \in (1, 2]$ and $f \in \mathrm{C}^1(\overline{\Omega})$ which is semiconcave in $\Omega$. Assume 
\begin{itemize}
    \item[\rm(1)] $f(x) \equiv \min_{y\in \overline{\Omega}}f(y)$ for all $ x \in \partial \Omega$,
    \item[\rm(2)] $f(x) > \min_{y\in \overline{\Omega}}f(y)$ for all $x \in \Omega$,
    \item[\rm(3)] there exists a constant $C>0$ such that $|Df(x)| \leq C \left(f(x)-\min_{x\in \overline{\Omega}}f(x)\right)^\frac{1}{p}$ for all $x \in \Omega$. 
\end{itemize} 
Let $x \in \Omega$. Then, there exists a minimizing curve $\xi(\cdot)$ for $x$ such that $T_{x, \xi} = +\infty$, that is, there exists a minimizing curve $\xi(\cdot)$ for every $x \in \Omega$ that does not escape the domain $\Omega$ in finite time.
\end{proposition}

\begin{remark}
We emphasize the difference between Proposition \ref{prp:lbtime} and Proposition \ref{prp:Tinfty} here. Proposition \ref{prp:lbtime} only states that for every $x \in \Omega$, there exists one minimizing curve $\xi$ such that $T_{x, \xi}=+\infty$, while Proposition \ref{prp:Tinfty} claims that for any minimizing curves emanating from $x \in \Omega$, the time to hit the boundary is $+\infty$. 
\end{remark}

\begin{proof}[Proof of Proposition \ref{prp:lbtime}]
Without loss of generality, we can assume $f=\min_{x\in \overline{\Omega}}f(x)\equiv0$ on $\partial \Omega$. From Lemma \ref{lem:f=0}, we have that $u=f=0$ on $\partial \Omega$.

We first show that for any $x\in \Omega$, there exists a minimizing curve $\xi$ for $u(x)$ such that $\xi$ is almost everywhere twice continuously differentiable in $[0, T_{x, \xi})$. Let $\xi$ be a minimizing curve for $u(x)$. If $\dot{\xi}(t) \neq 0$ for all $t \in (0, T_{x, \xi})$, then $\xi \in \mathrm{C}^2(0, T_{x, \xi})$ from Corollary \ref{cor:xic2}. Suppose $\dot{\xi}(t)=0$ for some $t \in (0, T_{x, \xi})$. Consider $t_0:=\inf\left\{t \geq 0: \dot{\xi}(t)=0\right\}$. Since $\xi\in \mathrm{C}^1\left[0,T_{x, \xi}\right)$, $\dot{\xi}(t_0)=0$. 

If $t_0=0$, consider a new curve $\tilde{\xi}(s)\equiv \xi(0)=x$, for all $s\in[0, +\infty)$. We claim that $\tilde{\xi}$ is another minimizing curve for $u(x)$. Indeed, compute the cost for $\tilde{\xi}$ and get
\begin{equation}\label{eqn:newcost}
    \int_0^{\infty}e^{-s}\left(C_p\left|\dot{\tilde{\xi}}(s)\right|^q+f\left(\tilde{\xi}(s)\right)\right)ds=  \int_0^{\infty}e^{-s} f(\xi(0)) ds=f(\xi(0)).
\end{equation}
From Proposition \ref{prp:ELderiH}, we know $u$ is differentiable at $\xi(s)$ for $s \in \left(0, T_{x, \xi}\right)$ and 
\begin{equation*}
    Du\left( \xi(s) \right)=-C_pq\left|\dot{\xi}(s)\right|^{q-2} \dot{\xi}(s).
\end{equation*}
Hence, $u$ solves \eqref{eqn:scpeq} classically at $\xi(s)$ for $s\in(0,T_{x, \xi})$ and
\begin{equation}\label{eqn:xidotdf-u}
C_pq \left|\dot{\xi}(s)\right|^{q-1}=\left|Du(\xi(s))\right|=(f(\xi(s))-u(\xi(s)))^\frac{1}{p}.
\end{equation}
Now, let $s \to 0$ and obtain
\begin{equation}
    C_pq\left|\dot{\xi}(0)\right|^{q-1}=(f(\xi(0))-u(\xi(0)))^\frac{1}{p}
\end{equation}
since $\xi \in \mathrm{C}^1\left[0, T_{x, \xi}\right)$ and $f, u$ are continuous. Hence, $\dot{\xi}(0) =0$ if and only if $f\left(\xi(0)\right)=u\left(\xi(0)\right)$.
Now, from \eqref{eqn:newcost}, the cost of $\tilde{\xi}$ is
\begin{equation}
     \int_0^{\infty}e^{-s}\left(C_p\left|\dot{\tilde{\xi}}(s)\right|^q+f\left(\tilde{\xi}(s)\right)\right)ds=f(\xi(0))=u(\xi(0))=u(x).
\end{equation}
Therefore, $\tilde{\xi} \equiv \xi(0) \in \mathrm{C}^2[0, +\infty)$ is a minimizing curve for $u(x)$.

If $t_0 >0$, consider a new curve
\begin{equation}
    \tilde{\xi}(s)=\left\{\begin{aligned}
      &\xi(s) \quad \text{ for } 0 \leq s 
      < t_0\\
      &\xi(t_0) \quad \text{ for } s \geq t_0
    \end{aligned}
    \right.
\end{equation}
Note that $\dot{\xi}(s)>0$ for $s \in [0, t_0)$. Hence, $\tilde{\xi}$ is twice continuously differentiable on $[0, t_0)$ by Corollary \ref{cor:xic2}. Apparently, $\tilde{\xi}$ is twice continuously differentiable on $(t_0, +\infty)$. Hence, $\tilde{\xi}$ is twice continuously differentiable except at $t_0$. Now we verify $\tilde{\xi}$ is a minimizing curve. By a similar argument from the previous case, we know 
\begin{equation*}
    f\left(\xi(t_0)\right)=u\left(\xi(t_0)\right)
\end{equation*}
since $\dot{\xi}(t_0)=0$. We can compute the cost of $\tilde{\xi}$ and get
\begin{equation}
\begin{aligned}
      \int_0^\infty e^{-s}\left(C_p\left|\dot{\tilde{\xi}}(s)\right|^q+f\left(\tilde{\xi}(s)\right)\right)ds &=\int_0^{t_0}e^{-s}\left(C_p\left|\dot{\xi}(s)\right|^q+f\left(\xi(s)\right)\right)ds+e^{-t_0}f\left(\xi(t_0)\right)\\
      &=\int_0^{t_0}e^{-s}\left(C_p\left|\dot{\xi}(s)\right|^q+f\left(\xi(s)\right)\right)ds+e^{-t_0}u\left(\xi(t_0)\right)\\
      &=u(x)
\end{aligned}
\end{equation}
where in the last equality, we use the dynamic programming principle and the fact that $\xi$ is a minimizing curve for $u(x)$. Therefore, $\tilde{\xi}$ is indeed a minimizing curve for $u(x)$.

We have proved that for any $x\in \Omega$, there exists a minimizing curve $\xi$ for $u(x)$ such that $\xi$ is almost everywhere twice continuously differentiable in $[0, T_{x, \xi})$. Next, we show for this minimizing curve, the time to escape the domain is $+\infty$.

If $T_{x, \xi}=+\infty$, we are done. Suppose $T_{x, \xi} < \infty$. From \eqref{eqn:xidotdf-u}, we see that
\begin{equation*}
C_pq \left|\dot{\xi}(s)\right|^{q-1}=\left|Du(\xi(s))\right|=(f(\xi(s))-u(\xi(s)))^\frac{1}{p} \to 0
\end{equation*}
as $s\to T_{x, \xi}$ since $f=u=0$ on $\partial \Omega$. This implies that as $s \to T_{x, \xi}$,
\begin{equation}
    \left|\dot{\xi}(s)\right|\to 0.
\end{equation}

Since $f>0$ in $\Omega$ and $f=0$ on $\partial \Omega$, there exists some constant $t_1 \in [0, T_{x, \xi})$ such that $$f\left(\xi(t_1)\right)=\max_{s\in \left[0, T_{x, \xi}\right)} f(\xi(s)) > 0.$$ From Proposition \ref{prp:ELderiH}, we know that the Euler-Lagrange equation is
\begin{equation}\label{eqn:ELnibp}
\begin{aligned}
  &-\left(e^{-s}C_pq\left| \dot{\xi}(s)\right|^{q-2} \dot{\xi}(s)\right)^\prime +e^{-s} Df\left( \xi (s)\right) =0\\
  \Longrightarrow & -\left(e^{-s}C_pq\left| \dot{\xi}(s)\right|^{q-2} \dot{\xi}(s)\right)^\prime \cdot \dot{\xi}(s) e^s+ Df\left( \xi (s)\right)\cdot \dot{\xi}(s) =0.
\end{aligned}
\end{equation}
Let $T \in (t_1, T_{x, \xi})$ and integrate both sides of \eqref{eqn:ELnibp} from $t_1$ to $T$ to get
\begin{equation}\label{eqn:ELibp}
\begin{aligned}
 &\int_{t_1}^T e^{-s}C_pq\left| \dot{\xi}(s)\right|^{q-2} \dot{\xi}(s)\left( \dot{\xi}(s) e^s\right)^\prime ds -C_pq\left| \dot{\xi}(s)\right|^q\bigg\rvert^T_{t_1} + \int_{t_1}^TDf\left( \xi (s)\right)\cdot \dot{\xi}(s) =0\\
\Longrightarrow &\int_{t_1}^T e^{-s}C_pq\left| \dot{\xi}(s)\right|^{q-2} \dot{\xi}(s)\left( \ddot{\xi}(s) e^s+\dot{\xi}(x)e^s\right) ds+C_pq\left| \dot{\xi}(t_1)\right|^q-C_pq\left| \dot{\xi}(T)\right|^q+ f(\xi(T))-f(\xi(t_1))=0.
\end{aligned}
\end{equation}
Let $T \to T_{x, \xi}$ in \eqref{eqn:ELibp} to obtain
\begin{equation}\label{eqn:ELibpa}
\begin{aligned}
 \int_{t_1}^{T_{x, \xi}}C_pq\left| \dot{\xi}(s)\right|^{q-2} \dot{\xi}(s)\ddot{\xi}(s)ds+\int_{t_1}^{T_{x, \xi}}C_pq\left| \dot{\xi}(s)\right|^qds +C_pq\left| \dot{\xi}(t_1)\right|^q&=f(\xi(t_1)) \\
  C_p\int_{t_1}^{T_{x, \xi}}\frac{d}{ds}\left(\left| \dot{\xi}(s)\right|^q\right)ds +\int_{t_1}^{T_{x, \xi}}C_pq\left| \dot{\xi}(s)\right|^qds +C_pq\left| \dot{\xi}(t_1)\right|^q&=f(\xi(t_1))\\
  \int_{t_1}^{T_{x, \xi}}C_pq\left| \dot{\xi}(s)\right|^qds +C_p(q-1)\left| \dot{\xi}(t_1)\right|^q&=f(\xi(t_1)).
\end{aligned}
\end{equation}
From \eqref{eqn:uderi}, compute
\begin{equation}\label{eqn:xidot}
     \left|\dot{\xi}(s)\right|=\left(\frac{1}{C_pq}\right)^\frac{1}{q-1}\left|Du(\xi(s))\right|^\frac{1}{q-1}=\left(\frac{1}{q^{-1}p^{-\frac{q}{p}}q}\right)^\frac{1}{q-1}\left|Du(\xi(s))\right|^\frac{1}{q-1}=p\left|Du(\xi(s))\right|^\frac{1}{q-1}.
\end{equation}
Plug \eqref{eqn:xidot} into \eqref{eqn:ELibpa} and get
\begin{equation*}
     \int_{t_1}^{T_{x, \xi}}C_p q p^q\left| Du \left(\xi(s)\right)\right|^p ds +C_p(q-1)p^q\left|Du \left(\xi(t_1)\right)\right|^p=f(\xi(t_1)).
\end{equation*}
Since $C_pqp^q=p$ and $C_p\left(q-1\right)p^q=1$, we obtain
\begin{equation}\label{eqn:fuu}
   \begin{aligned}
    \int_{t_1}^{T_{x, \xi}}p\left| Du \left(\xi(s)\right)\right|^p ds +\left|Du \left(\xi(t_1)\right)\right|^p&=f(\xi(t_1))\\
    \int_{t_1}^{T_{x, \xi}}p\left( f\left(\xi(s)\right)-u\left(\xi(s)\right)\right) ds +\left( f\left(\xi(t_1)\right)-u\left(\xi(t_1)\right)\right)&=f(\xi(t_1))\\
    \int_{t_1}^{T_{x, \xi}}p\left( f\left(\xi(s)\right)-u\left(\xi(s)\right)\right) ds &= u\left(\xi(t_1)\right),
\end{aligned}
\end{equation}
where to get the second line, we plug in $\left|Du\left(\xi(s)\right)\right|^p=f\left(\xi(s)\right)-u\left(\xi(s)\right)$ for $s \in (0, T_{x, \xi})$.

Now, for the left hand side of the last equation in \eqref{eqn:fuu}, we deduce
\begin{equation}\label{ieq:fu}
\begin{aligned}
      \int_{t_1}^{T_{x, \xi}}p\left( f\left(\xi(s)\right)-u\left(\xi(s)\right)\right) ds &\leq \int_{t_1}^{T_{x, \xi}}p(1-c_0) f\left(\xi(s)\right) ds\\
      &\leq \left(T_{x, \xi}-t_1\right)p(1-c_0) f(\xi(t_1))
\end{aligned}
\end{equation}
for some constant $c_0 \in \left(0, \frac{1}{2}\right]$ by Lemma \ref{lem:supsub} and the comparison principle, where the second inequality follows from the fact that $f\left(\xi(t_1)\right)=\max_{0\leq s \leq T_{x, \xi}}f(\xi(s))$.

For the right hand side of the last equation in \eqref{eqn:fuu}, we have
\begin{equation}\label{ieq:u}
    u\left(\xi(t_1)\right) \geq c_0f\left(\xi(t_1)\right)
\end{equation}
for some constant $c_0 \in \left(0, \frac{1}{2}\right]$ by Lemma \ref{lem:supsub} and the comparison principle again.

Combining \eqref{eqn:fuu}, \eqref{ieq:fu} and \eqref{ieq:u}, we obtain
\begin{equation}
\begin{aligned}
      c_0f\left(\xi(t_1)\right) &\leq \left(T_{x, \xi}-t_1\right)p(1-c_0) f(\xi(t_1))\\
      \frac{c_0}{p(1-c_0)}&\leq T_{x, \xi}-t_1 \\
      \frac{c_0}{p(1-c_0)}& \leq T_{x, \xi}.
\end{aligned}
\end{equation}
Note $c_0:=\min\left\{\frac{1}{2}, 2^{-1/p}\frac{1}{C}\right\}$ comes from Lemma \ref{lem:supsub}.
Therefore, $T_{x, \xi}\geq C_0$ for some constant $C_0$ that only depends on $p, C$ where $C$ comes from the condition $|Df| \leq Cf^\frac{1}{p}$. Now consider $\gamma(s):=\xi\left(s+T_{x, \xi}-\frac{C_0}{2}\right)$ for $s \in \left[0, \frac{C_0}{2}\right]$. Then, $\gamma(\cdot)$ is a minimizing curve for $u\left(\xi\left(T_{x, \xi}-\frac{C_0}{2}\right)\right)$. By the previous argument, we know $\displaystyle T_{\xi\left(T_{x, \xi}-\frac{C_0}{2}\right), \gamma} \geq C_0$. But this contradicts the fact that $\gamma(\frac{C_0}{2}) \in \partial \Omega$. Therefore, $T_{x, \xi}=+\infty$. 
\end{proof}

In the proof above, the key point is that there exists a minimizing curve for every $x\in \Omega$ such that it is twice differentiable almost everywhere, as is used in \eqref{eqn:ELibp}. In fact, we can modify the previous proof and show that for any minimizing curve $\xi$ emanating from any $x \in \Omega$, $T_{x, \xi}= +\infty$, even though $\xi$ might not be twice differentiable almost everywhere, which finally gives us the proof of Proposition \ref{prp:Tinfty}.

\begin{proof}[Proof of Proposition \ref{prp:Tinfty}]
If $T_{x, \xi}=+\infty$, we are done. Suppose $T_{x, \xi}<\infty$. We would like to proceed by contradiction, as in the proof of Proposition \ref{prp:lbtime}. It suffices to prove 
\begin{equation}
        \int_{t_1}^{T_{x, \xi}}p\left( f\left(\xi(s)\right)-u\left(\xi(s)\right)\right) ds = u\left(\xi(t_1)\right),
\end{equation}
where $t_1$ is chosen so that $f\left(\xi(t_1)\right)=\max_{0\leq s \leq T_{x, \xi}}f(\xi(s))$ as in \eqref{eqn:fuu}. Once we have this, we can argue by contradiction, exactly like in the proof of Proposition \ref{prp:lbtime}.

Suppose $\dot{\xi}(t) \neq 0$ at some point $t \in (0, T_{x, \xi})$. Since $\dot{\xi}$ is continuous, there exits $t_0 < t$ and $h > 0$ such that $\dot{\xi}(s) \neq 0$ for $s \in [t_0, t+h]$. From Corollary \ref{cor:xic2}, we know $\xi$ is twice continuously differentiable on $[t_0, t+h]$. Integrate \eqref{eqn:ELnibp} from $t_0$ to $t$ and obtain
\begin{equation}\label{eqn:nd}
\begin{aligned}
 &\int_{t_0}^t e^{-s}C_pq\left| \dot{\xi}(s)\right|^{q-2} \dot{\xi}(s)\left( \dot{\xi}(s) e^s\right)^\prime ds -C_pq\left| \dot{\xi}(s)\right|^q\bigg\rvert^t_{t_0} + \int_{t_0}^t Df\left( \xi (s)\right)\cdot \dot{\xi}(s) =0\\
\Longrightarrow &\int_{t_0}^t e^{-s}C_pq\left| \dot{\xi}(s)\right|^{q-2} \dot{\xi}(s)\left( \ddot{\xi}(s) e^s+\dot{\xi}(x)e^s\right) ds+C_pq\left| \dot{\xi}(t_0)\right|^q-C_pq\left| \dot{\xi}(t)\right|^q+ f(\xi(t))-f(\xi(t_0))=0.\\
       \Longrightarrow &\int_{t_0}^{t}C_pq\left| \dot{\xi}(s)\right|^qds+C_p(q-1)\left| \dot{\xi}(t_0)\right|^q-C_p(q-1)\left| \dot{\xi}(t)\right|^q+ f(\xi(t))-f(\xi(t_0))=0\\
      \Longrightarrow &\int_{t_0}^{t}C_pq\left| \dot{\xi}(s)\right|^qds=C_p(q-1)p^q\left[f(\xi(t))-u(\xi(t))-\left(f(\xi(t_0))-u(\xi(t_0))\right)\right]+f(\xi(t_0))-f(\xi(t)).
\end{aligned}
\end{equation}
If we take derivative with respect to $t$ on both sides of \eqref{eqn:nd}, we get
\begin{equation}\label{eqn:nda}
  C_pq\left| \dot{\xi}(t)\right|^q=C_p(q-1)p^q\left[Df(\xi(t))-Du(\xi(t))\right]\cdot \dot{\xi}(t)-Df(\xi(t)) \cdot \dot{\xi}(t).
\end{equation}
\eqref{eqn:nda} holds true at the time points $t$ where $\dot{\xi} (t)\neq 0$. Moreover, if $\xi$ is not twice differentiable at $t$, then $\dot{\xi}(t)=0$ by Corollary \ref{cor:xic2} and \eqref{eqn:nda} still holds true. Hence, \eqref{eqn:nda} is true for any $t \in [0, T_{x, \xi})$. Therefore, for any $t, T$ such that $0\leq t <T < T_{x, \xi}$,
\begin{equation}
    \int_{t}^{T}C_pq\left| \dot{\xi}(s)\right|^qds=C_p(q-1)p^q\left[f(\xi(T))-u(\xi(T))-\left(f(\xi(t))-u(\xi(t))\right)\right]+f(\xi(t))-f(\xi(T)).
\end{equation}
Let $T \to T_{x, \xi}$ and we obtain
\begin{equation}
\begin{aligned}
      &\int_{t}^{T_{x, \xi}}C_pq\left| \dot{\xi}(s)\right|^qds= -C_p(q-1)p^q\left(f(\xi(t))-u(\xi(t))\right)+f(\xi(t))\\
    \Longrightarrow &\int_{t}^{T_{x, \xi}}p\left(f(\xi(s))-u(\xi(s))\right)ds=u(\xi(t)).
\end{aligned}
\end{equation}
Then, we argue exactly as in the proof of Proposition \ref{prp:lbtime} and get a contradiction. Therefore, $T_{x, \xi}=+\infty$.
\end{proof}

Now, we are ready to prove the main result of global semiconcavity.
\begin{proof}[Proof of Theorem \ref{thm:main}]
Fix $x \in \Omega$ and let $\xi(\cdot)$ be a minimizing curve for $u(x)$. 
From Proposition \ref{prp:localsemi}, for any $T < T_{x, \xi}$, there exists a constant $C$ independent of $x, T$ such that
\begin{equation*}
    u(x+h)-2u(x)+u(x-h) \leq C \left(1+\frac{1}{T}\right) |h|^2
\end{equation*}
for all $h \in \mathbb{R}^n$ such that $[x-h, x+h] \subset \Omega$ with $|h| \leq M_{x,T}$ for some constant $M_{x,T}$ that depends on $x$ and $T$.

From Proposition \ref{prp:Tinfty}, since $T_{x, \xi} = +\infty$ for any $x\in \Omega$ and any minimizing curve $\xi$ for $u(x)$, we know $T_{x, \xi} \geq C_0$ for some constant $C_0>0$. Now take $\displaystyle T :=\frac{C_0}{2}$ and then,
\begin{equation*}
    u(x+h)-2u(x)+u(x-h) \leq C \left(1+\frac{2}{C_0}\right) |h|^2
\end{equation*}
for all $h \in \mathbb{R}^n$ such that $[x-h, x+h] \subset \Omega$ with $|h| \leq M_x$ for some constant $M_x := M_{x,C_0/2}$ that depends on $x$. Let $\phi$ be a smooth function in $\Omega$ such that $u-\phi$ has a local minimum at $x \in \Omega$. Then,
\begin{equation}\begin{aligned}
   \phi(x+h)-2\phi(x)+\phi(x-h) &\leq u(x+h)-2u(x)+u(x-h)\\
   & \leq C \left(1+\frac{2}{C_0}\right) |h|^2
\end{aligned}
\end{equation}
for all $h \in \mathbb{R}^n$ such that $[x-h, x+h] \subset \Omega$ with $|h| \leq M_x$.
Therefore, $D^2\phi(x)\leq \tilde{C}\mathrm{I}_n$ where $\displaystyle \tilde{C} = C \left(1+\frac{2}{C_0}\right)$. Hence, 
the solution $u$ to \eqref{eqn:scp} is a viscosity solution to 
\begin{equation*}
    -D^2u \geq -\tilde{C} \mathrm{I}_n \quad \text{ in } \Omega.
\end{equation*}
\end{proof}

In Theorem \ref{thm:main}, we showed that there exists a constant $\tilde{C}>0$ independent of $x$ such that for all $x \in \Omega$, 
\begin{equation}\label{eqn:key}
        u(x+h)-2u(x)+u(x-h) \leq \tilde{C} |h|^2
\end{equation}
for all $h\in\mathbb{R}^n$ such that $[x-h, x+h] \subset \Omega$ with $|h| \leq M_x$ for some constant $M_x$ that depends on $x$. However, by Definition \ref{def:semi}, in order to show $u$ is semiconcave in $\Omega$, we need \eqref{eqn:key} to hold for any $x, h \in \mathbb{R}^n$ such that $[x-h, x+h] \subset \Omega$. This is actually not hard to show. Let $v(x):=u(x)-\frac{\tilde{C}}{2}$. It is equivalent to show
\begin{equation}
    v(x+h)-2v(x)+v(x-h) \leq 0
\end{equation}
for any $x, h \in \mathbb{R}^n$ such that $[x-h, x+h] \subset \Omega$, i.e., $v$ is concave along any line segment contained in $\Omega$. From the fact that \eqref{eqn:key} holds with $h$ small enough, we know $v$ is locally concave. Local concavity does imply concavity along any line segment. We summarize the result in the following corollary.

\begin{corollary}
Under the same assumptions in Theorem \ref{thm:main}, the solution $u$ to \eqref{eqn:scp} satisfies
\begin{equation}
u(x+h)-2u(x)+u(x-h) \leq \tilde{C} |h|^2
\end{equation}
for any $x, h \in \mathbb{R}^n$ such that $[x-h, x+h] \subset \Omega$, where $\tilde{C}$ is the constant in Theorem \ref{thm:main}.
\end{corollary}

\section{One-dimensional results and examples}\label{sec:one}

In this section, we figure out the conditions in one-dimensional space under which the solution to \eqref{eqn:scp} is not semiconcave. In particular, we show the assumption (3) in Theorem \ref{thm:main} is essentially optimal, at least for the case $p=q=2$.

\begin{theorem} \label{thm:1d}
Let $1 < p \leq 2$ and $f \in \mathrm{C}([a,b]) \cap W^{1, \infty}((a,b))$ be a semiconcave function in $(a,b) \subset \mathbb{R}$ for some constants $a<b$. Assume $\displaystyle \min_{x \in [a,b]} f(x) =f(a)=f(b)$, and $f$ is differentiable in $(a, a+\epsilon) \cup (b-\epsilon, b)$ for some $\epsilon > 0$. Let $u$ be the solution to
\begin{equation}
\left\{\begin{aligned} \label{eqn:scp1d}
    u(x)+|u'(x)|^p -f(x) \leq 0 & \quad \text{in } (a,b), \\
    u(x)+|u'(x)|^p- f(x) \geq 0& \quad \text{on } [a, b].
    \end{aligned}
    \right.
\end{equation}

\begin{enumerate}
    \item[\rm(1)] If $f'(x) \equiv 0$, $\forall x \in (a, a+\epsilon) \cup (b-\epsilon, b)$, then $u$ is semiconcave in $(a, b)$.
    \item[\rm(2)] If $\, \displaystyle \frac{|f'(x)|}{\left(f(x) -f(b)\right)^\frac{1}{q}} \to +\infty$ as $x \to b$ (or $a$) where q satisfies $\frac{1}{p}+\frac{1}{q}=1$, then $u$ is not semiconcave.
\end{enumerate}
\end{theorem}
\begin{proof}
Without loss of generality, we can assume $\displaystyle \min_{[a,b]} f(x) =f(a)=f(b)=0$.
\begin{enumerate}
    \item If $f'(x) \equiv 0$, $\forall x\in (a, a +\epsilon) \cup (b-\epsilon,b)$, then $f(x) \equiv 0$ on $[a, a+\epsilon] \cup [b-\epsilon, b]$, which implies $u \equiv 0$ on $[a, a+\epsilon] \cup [b-\epsilon, b]$. Corollary \ref{cor:dis} then indicates $u$ is semiconcave on $[a, b]$.
    \item Without loss of generality, assume $\displaystyle \frac{\left|f'(x)\right|}{f(x)^\frac{1}{q}} \to +\infty$ as $x \to b$. Hence, $\left|f'(x)\right| >0$ on $(b-\epsilon_0, b)$ for some $\epsilon_0 \in (0, \epsilon)  $. In fact, $f'(x)<0$ on $(b-\epsilon_0, b)$ by Darboux's theorem since $f > 0$ on $(b-\epsilon_0, b)$. From Lemma \ref{lem:supsub}, we know $f$ is a supersolution to 
    \begin{equation}\label{eqn:1deq}
        u(x)+|u'(x)|^p -f(x)=0
    \end{equation} on $[a,b]$ and one can verify $v:=\min_{[a,b]}f(x)=0$ is a subsolution to \eqref{eqn:1deq} in $(a,b)$. By the comparison principle of \eqref{eqn:1deq}, we have $0\leq u \leq f$. 
    
    We claim $u < f$ on $(b-\epsilon_0, b)$. Suppose $u(x_0)=f(x_0)$ for some $x_0 \in (b-\epsilon_0, b)$. Using $f$ as a test function for $u$ at $x_0$, we get
    \begin{equation*}
    f(x_0)+|f'(x_0)|^p-f(x_0)>0 
    \end{equation*}
    which contradicts the fact that u is a subsolution to \eqref{eqn:scp1d}.
    
    Since $u$ is Lipschitz, hence $u$ is differentiable almost everywhere. Next, we show that $u$ is actually smooth near $x=b$. First, we prove that $u'=-\left(f-u\right)^\frac{1}{p}$ at points where $u$ is differentiable near $x=b$. Suppose there exists $x_0, x_1 \in (b-\epsilon_0, b)$ with $x_0 <x_1$ such that $u'(x_0)=-\left(f(x_0)-u(x_0)\right)^\frac{1}{p}$ and $u'(x_1)=\left(f(x_1)-u(x_1) \right)^\frac{1}{p}$. Choose $c$ such that $0 < c< \min_{x \in [x_0, x_1]} \left(f(x)-u(x)\right)^\frac{1}{p}$. Then $u(x) - cx$ attains a minimum at some point $x_2 \in (x_0, x_1)$. Using $cx$ as a test function, we get
    \begin{equation*}
        u(x_2)+c^p-f(x_2) < 0
    \end{equation*}
    which contradicts the fact of $u$ being a supersolution to \eqref{eqn:scp1d}. Therefore, either $u'=-\left(f-u\right)^\frac{1}{p}$ at all the points where $u$ is differentiable in $(b-\tilde{\epsilon}, b)$ for some $\tilde{\epsilon}>0$ or $u'=(f-u)^\frac{1}{p}$ at every point where $u$ is differentiable in $(b-\epsilon_0, b)$. If $u'=(f-u)^\frac{1}{p}$ a.e. in $(b-\epsilon_0, b)$, then 
    \begin{equation}
        u(x)=-\int_x^b(f(t)-u(t))^\frac{1}{p}dt <0
    \end{equation}
    which contradicts that $u \geq 0$. Hence, $u'=-\left(f-u\right)^\frac{1}{p}$ a.e. in $(b-\tilde{\epsilon}, b)$ for some $\tilde{\epsilon}>0$.
    Since $u$ is Lipschitz and $u'=-\left(f-u\right)^\frac{1}{p}$ a.e. in $(b-\tilde{\epsilon}, b)$, by the fundamental theorem of calculus, 
    \begin{equation*}
           u(x)=\int_x^b(f(t)-u(t))^\frac{1}{p}dt, 
    \end{equation*}
    for $x \in \left[b-\tilde{\epsilon}, b\right]$ and hence $u$ is smooth in $(b-\tilde{\epsilon}, b)$.
    
    If we differentiate \eqref{eqn:1deq} for $x \in (b-\tilde{\epsilon}, b)$, we have 
    \begin{equation}\label{eqn:udp}
    \begin{aligned}
             &u'+p\sign(u')|u'|^{p-1}u''-f'=0\\
             \Longrightarrow &u''=\frac{f'}{p\sign(u')|u'|^{p-1}}-\frac{1}{p}|u'|^{2-p}=\frac{f'}{-p(f-u)^\frac{p-1}{p}}-\frac{1}{p}|u'|^{2-p}.
    \end{aligned}
    \end{equation}
    Since $u$ is Lipschitz, $u \geq 0$, and $f'<0$ in $(b-\tilde{\epsilon}, b)$, we have 
    \begin{equation}
       \displaystyle u'' \geq \frac{-f'}{p(f-u)^\frac{1}{q}} -C \geq \frac{|f'|}{pf^\frac{1}{q}} - C \to +\infty \text{ as } x \to b.
    \end{equation}
    Therefore, $u''(x) \to +\infty$ as $x \to b$ and hence $u$ is not semiconcave.



\end{enumerate}
\end{proof}

The following is a specific example with $p=q=2$ where $f$ is semiconcave in the domain but the solution $u$ to \eqref{eqn:scp} is not globally semiconcave, which illustrates that the condition on $Df$ is necessary for the global semiconcavity of the constrained solutions. Also, it takes finite time for any minimizing curve to hit the boundary in this example.

\begin{example}\label{ex:nex}
Let $\Omega :=(-1, 1)$ and $f:[-1,1] \to \mathbb{R}$ defined by $f(x): =1-|x|$. Consider the solution $u$ to the equations
\begin{equation}\label{eqn:nex}
    \left\{\begin{aligned}
    u+\left|u'\right|^2-\left(1-|x|\right) & \leq 0 \quad \text{ in } (-1, 1), \\
    u+\left|u'\right|^2-\left(1-|x|\right) & \geq 0 \quad \text{ on } [-1, 1].
    \end{aligned}
    \right.
\end{equation}
Note that $f$ is semiconcave in $(-1, 1)$. Since $\displaystyle  \frac{|f'|}{|f|^\frac{1}{2}}=\frac{1}{\left(1-|x|\right)^\frac{1}{2}} \to +\infty $ as $|x| \to 1$, from part $(2)$ of Theorem \ref{thm:1d}, we know $u$ is not semiconcave and the semiconcavity constant blows up near $x=-1, 1$. 

More precisely, from part $(2)$ of Theorem \ref{thm:1d} and its proof, we know $u<f$ on $(-1,0) \cup (0,1)$. In fact, for this specific example, $u$ is smooth in $(-1,0) \cup (0,1)$. In particular, $u'=-(f-u)^\frac{1}{2}$ for $x \in (0, 1)$ and $u'=(f-u)^\frac{1}{2}$ for $x\in (-1, 0)$. 


Differentiate 
\begin{equation}\label{eqn:exeq}
    u+\left|u'\right|^2-\left(1-|x|\right) =0
\end{equation}in $(0, 1)$ to get
\begin{equation}\label{eqn:exudpp}
\begin{aligned}
      &u'+2u'u''+1=0\\
\Longrightarrow & u'' =\frac{-1-u'}{2u'}=\frac{1}{-2u'}-\frac{1}{2} \to +\infty 
\end{aligned}
\end{equation}
as $x \to 1^-$ since $u' \to 0^-$ as $x \to 1^-$.


Differentiate \eqref{eqn:exeq} in $(-1, 0)$ and obtain
\begin{equation}\label{eqn:exudpn}
    \begin{aligned}
      &u'+2u'u''-1=0\\
    \Longrightarrow & u'' =\frac{1-u'}{2u'}=\frac{1}{2u'}-\frac{1}{2} \to +\infty 
    \end{aligned}
\end{equation}
as $x \to -1^+$ since $u' \to 0^+$ as $x \to -1^+$. Hence, $u$ is not semiconcave. 

We claim that any minimizing curve for $u(x)$ with $x\in (-1,1)$ takes finite time to hit the boundary. Indeed, suppose there exists a minimizing curve $\gamma$ that takes infinite time to hit the boundary and $x_0:=\gamma(0) \in (-1, 1)$. We first prove that $|\gamma(s)| \to 1$ as $s \to +\infty$.

\begin{itemize}
    \item[\rm(1)] If $x_0:=\gamma(0)=0$, then $\gamma(\cdot)$ will not stay at $x=0$. We can see this by constructing a path that costs less than $\tilde{\gamma}(s): \equiv 0$. The cost for $\tilde{\gamma}(\cdot)$ is
    \begin{equation*}
         \int_0^\infty e^{-s} \left(\frac{\left|\dot{\tilde{\gamma}}(s)\right|^2}{4}+f(\tilde{\gamma}(s))\right)ds=\int_0^\infty e^{-s}f(0)ds= \int_0^\infty e^{-s}ds=1.
    \end{equation*}
    However, the cost for the path $\hat{\gamma}(\cdot)$ defined by
\begin{equation*}
\hat{\gamma}(s)=\left\{\begin{aligned}
   &s, \quad 0\leq s\leq 1\\    &1, \quad s \geq 1
    \end{aligned}
   \right.
\end{equation*}
is
\begin{equation*}
     \int_0^\infty e^{-s} \left(\frac{\left|\dot{\hat{\gamma}}(s)\right|^2}{4}+f(\hat{\gamma}(s))\right)ds=\int_0^1 e^{-s}\left(\frac{1}{4}+1-s\right)ds= \frac{1}{4}+\frac{3}{4}e^{-1}<1.
\end{equation*}
Therefore, the minimizing curve $\gamma$ will not stay at $x=0$.

\item[\rm(2)]If $\gamma(t) \in (0,1)$ for some $t\geq 0$, then 
\begin{equation}\label{eqn:exgdot}
\dot{\gamma}(t)=-2u'\left(\gamma(t)\right)=2\left(f\left(\gamma(t)\right)-u\left(\gamma(t)\right)\right)^\frac{1}{2} 
\end{equation}
by part $(2)$ of Proposition \ref{prp:ELderiH} and Lemma \ref{lem:c1}. Hence, $\dot{\gamma}(t)>0$.
Therefore, $\gamma$ is increasing. Hence, we have $\lim_{s\to \infty} \gamma(s)=1$. Since otherwise, suppose $\lim_{s \to \infty} \gamma(s)=x_1<1$. But $\dot{\gamma}(s) \geq \inf_{s\geq t} 2 \left(f\left(\gamma(s)\right)-u\left(\gamma(s)\right)\right)^\frac{1}{2} >0$, $\forall \, s \geq t$ and $\lim_{s \to \infty}\dot{\gamma}(s)=2\left(f\left(x_1)\right)-u\left(x_1\right)\right)^\frac{1}{2}>0$, which contradicts the assumption that $\lim_{s \to \infty} \gamma(s)=x_1<1$. Therefore, $\lim_{s\to \infty} \gamma(s)=1$.
\item[\rm(3)]If $\gamma(t) \in (-1,0)$ for some $t\geq0$, then by a similar argument as in part $(2)$, $\lim_{s\to \infty}\gamma(s)=-1$.
\end{itemize}



Therefore, $|\gamma(s)| \to 1$ as $s \to +\infty$. Without loss of generality, assume $\gamma(s) \to 1$ as $s \to +\infty$. Then, for every $x$ near $1$, there exists a minimizing curve for $u(x)$ that takes infinite time to hit the boundary. Proposition \ref{prp:localsemi} then implies $u''$ is bounded near $x=1$, which is a contradiction to \eqref{eqn:exudpp}.  
\end{example}

In Theorem \ref{thm:main}, we assume $f \equiv \min_{x\in[a,b]}f(x)$ on the boundary. The situation is more complicated when the minimum is only attained in the interior of the domain. It turns out that whether $u$ is semiconcave depends on the dynamics, which is illustrated in the following two examples.

\begin{example} \label{ex:compute}
Let $\Omega :=(-1, 1)$ and $f:[-1,1] \to \mathbb{R}$ defined by 
\begin{equation*}
    f(x):=\left\{\begin{aligned}
     &\left(|x|-\frac{1}{2}\right)^2 \quad \text{ in } \left[-1, -\frac{1}{2}\right]\cup \left[\frac{1}{2},1\right]\\
     &0 \qquad \qquad \qquad \text{ in } \left(-\frac{1}{2}, \frac{1}{2}\right).
    \end{aligned}
    \right.
\end{equation*}
Consider the solution $u$ to the equations
    \begin{equation}
        \left\{
  \begin{aligned}
    u(x) + \frac{1}{2}|u^\prime(x)|^2 -f(x) &\leq 0 \, \quad \qquad \text{in } (-1, 1), \\
              u(x) + \frac{1}{2}|u^\prime (x)|^2 -f(x) &\geq 0 \qquad  \quad \text{on } [-1, 1].
  \end{aligned}
\right.
    \end{equation}
The constant $\frac{1}{2}$ in front of $|u'|^2$ is just to make the computations easier and does not change the nature of the equations. In this example, we can explicitly compute the solution $u$ which is given by
\begin{equation*}
    u(x):=\left\{\begin{aligned}
     &\frac{1}{2}\left(|x|-\frac{1}{2}\right)^2 \quad \text{ in } \left[-1, -\frac{1}{2}\right]\cup \left[\frac{1}{2},1\right]\\
     &0 \qquad \qquad \qquad \text{ in } \left(-\frac{1}{2}, \frac{1}{2}\right),
    \end{aligned}
    \right.
\end{equation*}
and all the minimizing curves. Detailed computations are given in Appendix. Note in this case, all the minimizing curves for $x \in \left(-1, -\frac{1}{2}\right) \cup \left(\frac{1}{2},1\right)$ take infinite time to hit $x= -\frac{1}{2}, \frac{1}{2}$. In particular, for any minimizing curve $\xi(\cdot)$ for $u(x)$ with $x\in\left(-1, -\frac{1}{2}\right)$, $\lim_{s\to \infty} \xi(s)=-\frac{1}{2}$. Similarly, for any minimizing curve $\xi(\cdot)$ for $u(x)$ with $x\in\left(\frac{1}{2},1\right)$, $\lim_{s\to \infty} \xi(s)=\frac{1}{2}$.

As we can see, $f$ attains its minimum $0$ in the interior of $(-1,1)$ and $u$ is semiconcave.
\end{example}

\begin{example}
Let $\Omega :=(-1, 1)$. We first consider $f_1:[-1,1] \to \mathbb{R}$ defined by $f_1(x)=x^2$ and the equations
  \begin{equation}
        \left\{
  \begin{aligned}
    u(x) + \frac{1}{2}|u^\prime(x)|^2 -x^2 &\leq 0 \, \quad \qquad \text{in } (-1, 1), \\
              u(x) + \frac{1}{2}|u^\prime (x)|^2 -x^2 &\geq 0 \qquad  \quad \text{on } [-1, 1].
  \end{aligned}
\right.
    \end{equation}
In this case, one can verify the solution $u$ is $\displaystyle u(x)=\frac{1}{2}x^2$.
Now we change $f_1$ a little bit near $x=-1$ and $1$. Consider $f_2:[-1,1] \to \mathbb{R}$ defined by
\begin{equation*}
    f_2(x):=\left\{\begin{aligned}
     &x^2 \qquad \qquad \quad \text{ in } \left[-\frac{1}{2}, \frac{1}{2}\right],\\
     &-\frac{|x|}{4}+\frac{3}{8}  \qquad \text{ in } \left[-1, -\frac{1}{2}\right) \cup  \left(\frac{1}{2}, 1\right].
    \end{aligned}
    \right.
\end{equation*}
Note that $f_2$ attains its minimum only at $x=0$. We claim that the solution $u$ to
  \begin{equation}
  \label{eqn:exp3}
        \left\{
  \begin{aligned}
    u(x) + \frac{1}{2}|u^\prime(x)|^2 -f_2(x)&\leq 0 \, \quad \qquad \text{in } (-1, 1), \\
              u(x) + \frac{1}{2}|u^\prime (x)|^2 -f_2(x) &\geq 0 \qquad  \quad \text{on } [-1, 1],
  \end{aligned}
\right.
    \end{equation}
is not semiconcave.

The following are some observations.
\begin{itemize}
    \item Near $x=0$, we have $u'=\sqrt{2(f_2-u)}$ at the points where $u$ is differentiable, that is, $u'=\sqrt{2(f_2-u)}$ almost everywhere near $x=0$. This is because $u>0$ at every point in $[-1,1]$ except at $x=0$. By a similar argument as in the proof of part $(2)$ of Theorem \ref{thm:1d}, we can show $f_2>u$ in $\displaystyle \left(0,\frac{1}{2}\right)$ since $f_2'\neq0$. Moreover, we cannot have $u'(x_0)<0$ and $u'(x_1)>0$ for some $x_0<x_1$. Hence, $u'=\sqrt{2(f_2-u)}>0$ almost everywhere in $(-\epsilon, \epsilon)$ for some $\epsilon>0$. Therefore, $u$ is smooth in $(-\epsilon, \epsilon)$ and $u(x)=\frac{1}{2}x^2$ in $(-\epsilon, \epsilon)$.
    \item $u'=\sqrt{2(f_2-u)}$ almost everywhere in $\left[0,\frac{1}{2} \right)$ and is smooth in $\left[0,\frac{1}{2} \right)$. Since otherwise, suppose $u'=-\sqrt{2(f_2-u)}$ at $x_0 \in \left[0, \frac{1}{2}\right)$. Then, $u(x_0) \leq \frac{1}{2}x_0^2 \leq \frac{1}{8}$. Again, because we cannot have $u'(x_0)<0$ and $u'(x_1)>0$ for some $x_0<x_1$, $u'=-\sqrt{2(f_2-u)}$ on $[x_0,1]$. Then, $u(1) <\frac{1}{8}$ and $u'(1)<0$, which contradicts the fact that $u$ is a supersolution to \eqref{eqn:exp3} on $[-1, 1]$.
    \item $u'=-\sqrt{2(f_2-u)}$ at some point $ x_1 \in \left(\frac{1}{2}, 1\right)$ where $u$ is differentiable. Suppose otherwise, that is, $u'=\sqrt{2(f_2-u)}$ at every point where $u$ is differentiable in $\left(\frac{1}{2}, 1\right)$. By a similar argument as in the proof of part $(2)$ of Theorem \ref{thm:1d}, we can show $f_2>u$ in $ \left(\frac{1}{2},1\right)$ since $f_2'\neq0$. But if we run $u$ according to $u'=\sqrt{2(f_2-u)}>0$ from $x=\frac{1}{2}$, then $u$ and $f_2$ must intersect at some point $x_2 \in \left(\frac{1}{2}, 1\right)$ as $u\left(\frac{1}{2}\right)=\frac{1}{8}$ and $f(1)=\frac{1}{8}$. This contradicts the fact that $f_2>u$ in $\left(\frac{1}{2},1\right)$.
    \item Since $u'=-\sqrt{2(f_2-u)}$ at some point $ x_1 \in \left(\frac{1}{2}, 1\right)$ where $u$ is differentiable, $u'=-\sqrt{2(f_2-u)}$ at every point $x \in [x_1,1)$ where $u$ is differentiable, i.e., $u$ is smooth near $x=1$ and $u'=-\sqrt{2(f_2-u)}$. Moreover, $u' \to 0^-$ as $x \to 1^-$ since otherwise it contradicts the fact that $u$ is a subsolution on $[-1, 1]$. Near $x=1$, we have
    \begin{equation*}
        u''=\frac{f'_2-u'}{u'}=\frac{1}{4\sqrt{2\left(-\frac{x}{4}+\frac{3}{8}-u\right)}}-1 \to +\infty
    \end{equation*} 
    as $x \to 1^-$. We can get the same conclusion near $x=-1$ by similar arguments.
    
\end{itemize}
In summary, for this example, the solution $u$ is not semiconcave and the semiconcavity constant blows up as $|x| \to 1$ .
\end{example}

\renewcommand\thesection{A}

\section{Appendix}

\subsection{Existence of the solution to \eqref{eqn:scp}}

\begin{theorem}
Let $\Omega \subset \mathbb{R}^2$ be an open, bounded, and connected domain. Let $H :\overline{\Omega} \times \mathbb{R}^n \to \mathbb{R
}$ satisfy (H1). Then there exists a constrained viscosity solution $u \in \mathrm{C}(\overline{\Omega})\cap \mathcal{W}^{1,\infty}(\Omega)$ to \eqref{eqn:scp}.
\end{theorem}

\begin{proof}
Since $\Omega$ is bounded, we can deduce from $(H1)$ (or coercivity of $H$ instead) that $\max\{-H(x,p): (x, p) \in \overline{\Omega}\times \mathbb{R}^n\} \leq C_1$ for some constant $C_1 > 0$. Similarly, we know there exists a constant $C_2 >0$ such that $\sup_{x \in \Omega} |H(x,0)| \leq C_2$.
Again from $(H1)$ (or coercivity of $H$ instead), we can find a constant $C_3>0$ such that
whenever $H(x,p) \leq \max\{C_1, C_2\}$ for some $x \in \overline{\Omega}$, we have $|p| <C_3$.

Define $\mathcal{A}: = \{v \in \mathrm{C}(\overline{\Omega})\cap \mathcal{W}^{1,\infty}(\Omega) : -C_2 \leq v(x) \leq C_1, \|Dv\|_{L^\infty} \leq C_3$, and $v$ is a viscosity solution to $v(x)+H(x, Dv) \leq 0$ in $ \Omega. \}$ and $u(x):= \sup\{v(x): v\in \mathcal{A}\}$. It is not hard to show $u$ is a viscosity subsolution to \eqref{eqn:scpeq} in $\Omega$. It remains to show $u$ is a supersolution to \eqref{eqn:scpeq}on $\overline{\Omega}$.

Suppose $u$ is not a supersolution on $\overline{\Omega}$. Then there exists $x_0 \in \overline{\Omega}$, $r>0$, and $\phi \in \mathrm{C}^1 (\overline{\Omega})$ such that 
\begin{itemize}
    \item[(a)] $u(x_0) = \phi(x_0)$,
    \item[(b)] $\phi (x_0) +H(x_0, D\phi(x_0)) <0$,
    \item[(c)] $-C_2-1 \leq \phi(x) < C_1, \forall x\in \mathrm{B}(x_0, r)\cap \overline{\Omega}$,
    \item[(d)] $u(x)-\phi(x) \geq |x-x_0|^2$, $\forall x \in \mathrm{B}(x_0,r) \cap \overline{\Omega}$,
    \item[(e)] $\|D\phi(x)\| \leq C_3, \forall x \in \mathrm{B}(x_0,r)\cap \overline{\Omega}$.
\end{itemize}  
(c) is possible because $\phi(x_0) < -H(x_0,D\phi(x_0)) \leq C_1$.
    Choose $\epsilon$ small enough so that $\forall x \in \mathrm{B}(x_0, 2\epsilon)$,
\begin{equation*}
    \left\{\begin{aligned}
    &\phi(x)+\epsilon^2<C_1,\\
    &\phi(x)+\epsilon^2+H(x, D\phi(x)) <0.
    \end{aligned}
    \right.
\end{equation*}
Hence, $\phi + \epsilon^2$ is a subsolution to \eqref{eqn:scpeq} in $\mathrm{B}(x_0,2\epsilon)\cap \overline{\Omega}$. From (c) above, $u(x) > \phi(x) +\epsilon^2$, $\forall x \in \mathrm{B}(x,2\epsilon) \setminus \mathrm{B}(x, \epsilon)$.

Define \begin{equation}
   v(x):= \left\{\begin{aligned}
    &\max\{u(x),\phi(x)+\epsilon^2\}, \quad \text{if } x \in \mathrm{B(x_0,2\epsilon)} \cap \overline{\Omega},\\
    &u(x), \qquad \qquad \qquad \qquad \quad \text{if } x \in \overline{\Omega}\setminus \mathrm{B(x_0, 2\epsilon)}.
    \end{aligned}
    \right.
\end{equation}
We see $v$ is a subsolution to \eqref{eqn:scpeq} in $\Omega$, which is larger than $u$. This contradicts the definition of $u$ and therefore, $u$ is a supersolution to \eqref{eqn:scpeq} on $\overline{\Omega}$.
\end{proof}

\subsection{Regularity of minimizing curves}
\begin{proof}[Proof of Lemma \ref{lem:c1}]

The optimal control formula for the solution $u$ is
\begin{equation}
    u(x) = \inf \left\lbrace \int_0^\infty e^{-s}\left(C_p|\dot{\gamma}(s)|^{q} +f(\gamma(s))\right)ds: \gamma\in \mathrm{AC}([0,\infty);\overline{\Omega}), \gamma(0) = x\right\rbrace,
\end{equation}
where 
\begin{equation*}
    C_p = q^{-1} p^{-\frac{q}{p}} \qquad\text{and}\qquad \frac{1}{p} + \frac{1}{q} = 1.
\end{equation*}

Fix $b >0$ such that $\xi(s) \in \Omega, \forall \, s \in [0,b]$, and define $$\displaystyle I_b[\gamma] := \int_0^b e^{-s} \Big( C_p|\dot{\gamma}(s)|^{q} +f(\gamma(s)) \Big)ds.$$
Since $\xi$ is a minimizer of \eqref{eqn:opf}, $\xi$ is also a minimizer of the problem
\begin{equation}
    \min \left\{ I_b [\gamma] ; \gamma \in \mathrm{AC}\left([0, b]; \Omega \right) , \gamma(0) = x, \gamma(b) =\xi(b) \right\}.
\end{equation}
Let $\zeta \in \mathrm{C}^\infty\left([0, b]; \mathbb{R}^n \right)$ with $ \zeta(0) = 0, \zeta(b) =0$ and $\tau \in \mathbb{R}$. Hence, $\xi(s) +\tau \zeta(s) \in \Omega$, $ \forall s \in [0,b]$, for $|\tau|$ small enough.

By calculus of variation, we know
\begin{equation*}
\begin{aligned}
    0&=\left.\frac{d}{d\tau}\left( I[\xi +\tau \zeta]\right) \right|_{\tau=0}\\
    &=\int_0^b e^{-s} \left( C_p q \left| \dot{\xi}(s)\right|^{q-2} \dot{\xi} (s) \cdot \dot{\zeta}(s) +Df\left( \xi (s)\right) \cdot \zeta (s) \right) ds,\\
\end{aligned}
\end{equation*}
which is true for any $\zeta \in \mathrm{C}^\infty\left([0, b]; \mathbb{R}^n \right)$ with $ \zeta(0) = 0, \zeta(b) =0$.

Therefore,
\begin{equation} \label{eqn:ELc1}
    e^{-t} C_p q \left| \dot{\xi}(t)\right|^{q-2} \dot{\xi}(t) = \int_0^te^{-s} Df\left(\xi(s)\right)ds+\tilde{c}
\end{equation}
for a.e. $t \in [0,b]$.


$\forall \, t\in [0,b]$ and $h>0$ small, from the optimal control formula of $u(x)$, we have
\begin{equation}
    \frac{e^{-t}u\left(\xi(t)\right)-e^{-(t+h)}u(\xi(t+h))}{h}=\frac{1}{h}\int_t^{t+h}e^{-s}C_p\left|\dot{\xi}(s)\right|^{q} ds.
\end{equation}
Let $\phi \in \mathrm{C}^1(\mathbb{R}^n)$ such that $u-\phi$ has a local minimum at $\xi(t)$ and $u\left(\xi(t)\right)=\phi\left(\xi(t)\right)$.
Then,
\begin{equation}
    \begin{aligned}
     &\frac{e^{-t}u\left(\xi(t)\right)-e^{-(t+h)}u(\xi(t+h))}{h} \leq \frac{e^{-t}\phi\left(\xi(t)\right)-e^{-(t+h)}\phi(\xi(t+h))}{h}\\
     \Longrightarrow & \frac{1}{h}\int_t^{t+h}e^{-s}C_p\left|\dot{\xi}(s)\right|^{q} ds \leq \frac{e^{-t}\phi\left(\xi(t)\right)-e^{-(t+h)}\phi(\xi(t+h))}{h}
    \end{aligned}
\end{equation}
Since $\xi$ is differentiable a.e. in $(0, +\infty)$, at the point where $\xi$ is differentiable, let $h \to 0^+$ and get
\begin{equation}
    \begin{aligned}
      &e^{-t}C_p\left|\dot{\xi}(t)\right|^{q} \leq e^{-t}\phi\left(\xi(t)\right)-e^{-t}D\phi\left(\xi(t)\right) \cdot \dot{\xi}(t)\\
      \Longrightarrow & C_p\left|\dot{\xi}(t)\right|^{q} \leq \phi\left(\xi(t)\right)-D\phi\left(\xi(t)\right) \cdot \dot{\xi}(t).
    \end{aligned}
\end{equation}

Since there exists a constant $C>0$ such that
\begin{equation}
    ||u||_{L^{\infty}}+||Du||_{L^{\infty}} \leq C,
\end{equation}
we have 
\begin{equation}
    \left|\dot{\xi}(t)\right|^{q} \leq C+C  \left|\dot{\xi}(t)\right|
\end{equation}
for some constant $C>0$ independent of the curve $\xi$ and time $t$. Since $q \geq 2$, $ \left\|\dot{\xi}\right\|_{L^{\infty}} \leq C$ for some constant $C>0$. This bound is independent of the specific curve $\xi$. Therefore, $\xi$ is Lipschitz. 

To see why $\xi$ is $C^1$, note that 
$$D_vL(x,v) = e^{-t}C_pq|v|^{q-2}v$$
is injective in $v$. Together with \eqref{eqn:ELc1}, we can show $\xi \in \mathrm{C}^1([0,b])$. See Theorem 6.2.5 in \cite{cannarsa2004semiconcave} for details.
\end{proof}

\subsection{An example of solving the first-order state-constraint equation in 1D}

Let u be the solution to
    \begin{equation}
        \left\{
  \begin{aligned}
    u(x) + \frac{1}{2}|u^\prime(x)|^2 -f(x) &\leq 0 \, \quad \qquad \text{in } (-1, 1), \\
              u(x) + \frac{1}{2}|u^\prime (x)|^2 -f(x) &\geq 0 \qquad  \quad \text{on } [-1, 1],
  \end{aligned}
\right.
    \end{equation}
    where $f$ is defined by
\begin{equation*}
    f(x):=\left\{\begin{aligned}
     &\left(|x|-\frac{1}{2}\right)^2 \quad \text{ in } \left[-1, -\frac{1}{2}\right]\cup \left[\frac{1}{2},1\right]\\
     &0 \qquad \qquad \qquad \text{ in } \left(-\frac{1}{2}, \frac{1}{2}\right).
    \end{aligned}
    \right.
\end{equation*}
    We solve this equation by deterministic optimal control formula and Euler-Lagrange equation.
    
\textbf{STEP 1. Find the Euler-Lagrange equation.}

According to the optimal control formula for the state-constraint problem,
\begin{equation}
\label{ocf}
    u(x)=\inf \left\{ \left.I[\gamma]: =\int_0^\infty e^{-s} \left( \frac{\left( \dot{\gamma}(s) \right)^2}{2} + f(\gamma (s)) \right)ds \right| \gamma \in \mathrm{AC}([0, \infty); [-1,1]), \gamma(0)=x \right\}.
\end{equation}
Let $x_0 \in (-1, 1)$. Suppose a minimizer of \eqref{ocf} with initial data $x_0$ exits and call it $\xi$. Fix $b >0$ such that $\xi(s) \in (-1, 1), \forall s \in [0,b]$, and define $$\displaystyle I_b[\xi] := \int_0^b e^{-s} \left( \frac{\left( \dot{\xi}(s) \right)^2}{2} + f(\xi (s)) \right)ds.$$
Since $\xi$ is a minimizer of \eqref{ocf}, $\xi$ is also a minimizer of the problem
\begin{equation}
    \min \left\{ I_b [\gamma] ; \gamma \in \mathrm{AC}\left([0, b]; [-1,1]\right) , \gamma(0) = x_0, \gamma(b) =\gamma(b) \right\}.
\end{equation}
Let $\zeta \in \mathrm{C}^\infty\left([0, b]; [-1,1]\right)$ with $ \zeta(0) = 0, \zeta(b) =0$ and $\tau \in \mathbb{R}$. Hence, $\xi(s) + \tau \zeta(s) \in (-1, 1)$, $ \forall s \in [0,b]$, for $|\tau|$ small enough. By calculus of variation, 
\begin{equation*}
\begin{aligned}
    \left.  0=\frac{d}{d\tau}\left( I[\xi +\tau \zeta]\right) \right|_{\tau=0} &=\int_0^b e^{-s} \left(\dot{\xi}(s)  \dot{\zeta}(s) + f^\prime(\xi(s)) \zeta(s) \right)ds\\
    &=\int_0^b -\left( e^{-s} \dot{\xi}(s) \right)^\prime \zeta(s) + e^{-s} f^\prime(\xi(s)) \zeta(s)ds\\
    &=\int_0^b \left( \dot{\xi} -\Ddot{\xi} +f^\prime (\xi)\right)\zeta e^{-s}ds.
\end{aligned}
\end{equation*}
Therefore, if $\xi$ is a minimizer of \eqref{ocf}, then 
\begin{equation}
\label{eqn:aEL}
\left\{
\begin{aligned}
     \Ddot{\xi}(s)&=\dot{\xi}(s) + f^\prime (\xi(s)), s\in (0,b),\\
     \xi(0)&=x_0.
\end{aligned}
\right.
\end{equation}
Note that we don't have any info about initial velocity here.

We will instead consider \eqref{eqn:aEL} with the initial position $x_0$ and the initial velocity $y_0$ (to be determined), i.e,
\begin{equation}
\label{eqn:aELv}
\left\{
\begin{aligned}
     \Ddot{\xi}(s)&=\dot{\xi}(s) + f^\prime (\xi(s)), s\in (0,b),\\
     \xi(0)&=x_0,\\
     \dot{\xi}(0)&=y_0.
\end{aligned}
\right.
\end{equation}
Change notations by letting $x(s)= \xi(s)$ and $y(s)=\dot{\xi}(s)$. Then \eqref{eqn:aELv} becomes
\begin{equation}
\label{eqn:ode}
\left\{
\begin{aligned}
      \dot{x}(s)&=y(s),\\
     \dot{y}(s)&=y(s) + f^\prime (x(s)),\\
     x(0)&=x_0,\\
     y(0)&=y_0.
\end{aligned}
\right.
\end{equation}

\textbf{STEP 2. Solve the Euler-Lagrange equation to get the formula for $\xi$ and compute the cost function value $I[\xi]$.} 


    
\begin{enumerate}
    \item Suppose $x_0 \in [-\frac{1}{2}, \frac{1}{2}]$. We can choose $\xi (s) \equiv x_0$. Then $I[\xi]=0$ and hence $\xi \equiv x_0$ is a minimizing curve. Thus, $u(x)=0$ for $x \in [-\frac{1}{2}, \frac{1}{2}]$.
    \item Suppose $x_0 \in (\frac{1}{2}, 1]$. We need to solve the Euler-Lagrange equation.

If $x(s) \in \left[\frac{1}{2}, 1\right]$, we have
\begin{equation*}
\begin{aligned}
 &\left\{
\begin{aligned}
\dot{x}(s)&=y(s),\\
     \dot{y}(s)&=y(s) + 2\left(x(s)-\frac{1}{2}\right),\\
     x(0)&=x_0,\\
     y(0)&=y_0,
\end{aligned}
\right.\\
\Longrightarrow &
\left\{
    \begin{aligned}
         	& x(s)=\frac{1}{2}+\left(\frac{x_0}{3}+\frac{y_0}{3}-\frac{1}{6}\right)e^{2t} + \left( \frac{2}{3}x_0-\frac{1}{3}y_0-\frac{1}{3}\right) e^{-t},\\
         	& y(s)=2\left(\frac{x_0}{3}+\frac{y_0}{3}-\frac{1}{6}\right)e^{2t} - \left( \frac{2}{3}x_0-\frac{1}{3}y_0-\frac{1}{3}\right) e^{-t}.
    \end{aligned}
\right.\\
\end{aligned}
\end{equation*}
To simplify the expression, we let 
\begin{equation*}
\begin{aligned}
     A &:=\frac{x_0}{3}+\frac{y_0}{3}-\frac{1}{6},\\
    B &:= \frac{2}{3}x_0-\frac{1}{3}y_0-\frac{1}{3},\\
\end{aligned}
\end{equation*}
and then
\begin{equation}
    \left\{
    \begin{aligned}
         	& x(s)=\frac{1}{2}+Ae^{2t} + B e^{-t},\\
         	& y(s)=2Ae^{2t} - B e^{-t}.
    \end{aligned}
\right.
\end{equation}
Set $\displaystyle x(t)=\frac{1}{2}$ and we can try to solve for the time $t_\frac{1}{2}$ when the position is at $x=\frac{1}{2}$:
\begin{equation*}
        \begin{aligned}
         &\frac{1}{2}+Ae^{2t} + B e^{-t} =\frac{1}{2},\\
         	\Longrightarrow \quad & t_\frac{1}{2}=\frac{1}{3} \log \left(-\frac{B}{A}\right).\\
    \end{aligned}
\end{equation*}

$t_{\frac{1}{2}}$ may or may not exist, depending on whether $\displaystyle -\frac{B}{A}$ is positive or not. We break into four cases according to the initial velocity $y_0$ as follows.
\begin{enumerate}
    \item Suppose $t_\frac{1}{2}$ exists. $\displaystyle -\frac{B}{A} > 0 , y_0 <0 \Rightarrow A<0 \Rightarrow y_0 < \frac{1-2x_0}{2}$. In this case, after the curve reaches $x=\frac{1}{2}$, we will let the curve stay at $x=\frac{1}{2}$ to minimize the cost.
    \begin{equation*}
    \begin{aligned}
        I[\xi]&= \int_0^\infty e^{-s}\left( \frac{y(s)^2}{2} +\left(x(s)-\frac{1}{2}\right)^2\right)ds\\
        &=\int_0^\infty e^{-s}\left( \frac{(2Ae^{2s}-Be^{-s})^2}{2} +(Ae^{2s}+Be^{-s})^2\right)ds\\
        &=\int_0^{t_\frac{1}{2}} 3A^2e^{3s}+\frac{3}{2}B^2e^{-3s}ds\\
        &=\left. A^2e^{3s}-\frac{B^2}{2}e^{-3s} \right|^{\frac{1}{3}\log \left(-\frac{B}{A}\right)}_0 \\
        &=-\frac{x_0y_0}{2}+\frac{1}{4}y_0\\
        \Longrightarrow \frac{\partial I[\xi]}{\partial y_0} &= -\frac{x_0}{2}+\frac{1}{4}<0.\\
    \end{aligned}
    \end{equation*}
    Therefore, $I[\xi]$ is decreasing with respect to $y_0$ when $\displaystyle y_0 < \frac{1-2x_0}{2}$. 
    
    \item Suppose $A=0$, i.e., $\displaystyle y_0=\frac{1-2x_0}{2}$. We have
    \begin{equation*}
     \left\{
        \begin{aligned}
        x(s)&=\frac{1}{2}+\left(x_0-\frac{1}{2}\right)e^{-s},\\
        y(s)&=-\left(x_0-\frac{1}{2}\right)e^{-s}.
        \end{aligned}
    \right.
    \end{equation*}
    As we can see here, the curve never reaches $x=\frac{1}{2}$. Compute the cost function
    \begin{equation*}
        \begin{aligned}
         I[\xi] &= \int_0^\infty e^{-s}\left( \frac{(x_0-\frac{1}{2})^2e^{-2s}}{2} +(x_0-\frac{1}{2})^2e^{-2s}\right)ds\\
          &=\frac{1}{2}x_0^2-\frac{1}{2}x_0+\frac{1}{8}.\\
        \end{aligned}
    \end{equation*}
    
    \item Suppose $A >0, B \geq 0$, i.e., $\displaystyle \frac{1-2x_0}{2} < y_0 \leq 2x_0-1$. We claim that in this case we can always find a curve $\gamma$ that $I[\gamma] < I[\xi]$. 
    
    If we start at $x_0 \in (\frac{1}{2}, 1)$ and evolve according to the Euler-Lagrange equation, the velocity $\dot{\xi}(s) = y(s)$ may change sign as time goes, for instance, from negative to positive. We can consider another curve $\gamma(s)$ such that $\gamma(0)=x_0$ and $\dot{\gamma}(s)=-|\dot{\xi}(s)|$ for $s\in \left\{s \in [0, \infty) :\xi(s)=x(s)<1\right\}$ and stay at $x=\frac{1}{2}$ once the curve $\gamma(s)=\frac{1}{2}$. If $\xi$ arrives at $1$ before $\gamma$ hits $\frac{1}{2}$, there are two possibilities for $\xi$ afterwards. The first possibility is that $\xi$ can stay at $1$. Then we can choose $\gamma$ so that $\displaystyle \frac{\dot{\gamma}(s)^2}{2}+\left(\gamma(s)-\frac{1}{2}\right)^2<f(1)$ to make $I[\gamma] < I[\xi]$. The second possibility is that $\xi$ can move away from $1$. Then the strategy is like before, namely that we choose $\dot{\gamma}(s)=-|\dot{\xi}(s)|$ to get $I[\gamma] < I[\xi]$. 
    
    If we start at $x_0=1$, we can choose to stay at $x=1$, i.e., $\gamma (s) \equiv 1$. Then $\displaystyle I[\gamma]=\frac{1}{4}$. If we choose $\displaystyle y_0=\frac{1-2x_0}{2}$ as in the Case (b), $\displaystyle I[\xi]=\frac{1}{8}$.
    \item Suppose $y_0 > 2x_0-1$. In this case, $A >0, B <0$, and $y(s) = 2Ae^{2s}- Be^{-s} >0$. Hence, the curve $\xi$ will run towards $x=1$. Again, the strategy is to make $\dot{\gamma}(s)=-\dot{\xi}(s)$ and the argument is similar to that in the Case (c).
\end{enumerate}
In summary, the minimizing curve is $\displaystyle \xi (s) = \frac{1}{2}+\left( x_0 -\frac{1}{2}\right) e^{-s}$, which is the solution to the Euler-Lagrange equations with the initial velocity $\displaystyle y_0=\frac{1-2x_0}{2}$. Hence, $\displaystyle u(x)= \frac{1}{2}x^2-\frac{1}{2}x+\frac{1}{8}$ for $x \in (\frac{1}{2}, 1]$.
    \item Suppose $x_0 \in \left[-1, -\frac{1}{2}\right]$.
    If $x(s) \in \left[-1, -\frac{1}{2}\right]$, we have
    
\begin{equation*}
\begin{aligned}
 &\left\{
\begin{aligned}
\dot{x}(s)&=y(s),\\
     \dot{y}(s)&=y(s) + 2\left(x(s)+\frac{1}{2}\right),\\
     x(0)&=x_0,\\
     y(0)&=y_0,
\end{aligned}
\right.\\
\Longrightarrow &
\left\{
    \begin{aligned}
         	& x(s)=-\frac{1}{2}+\left(\frac{x_0}{3}+\frac{y_0}{3}+\frac{1}{6}\right)e^{2t} + \left( \frac{2}{3}x_0-\frac{1}{3}y_0+\frac{1}{3}\right) e^{-t},\\
         	& y(s)=2\left(\frac{x_0}{3}+\frac{y_0}{3}+\frac{1}{6}\right)e^{2t} - \left( \frac{2}{3}x_0-\frac{1}{3}y_0+\frac{1}{3}\right) e^{-t}.
    \end{aligned}
\right.\\
\end{aligned}
\end{equation*}
Similarly, we can show the minimizing curve is $\displaystyle \xi(s)=-\frac{1}{2}+\left(x_0+\frac{1}{2}\right)e^{-s}$ and $\displaystyle u(x)=\frac{1}{2}x^2+\frac{1}{2}x+\frac{1}{8}$ for $x \in \left[-1, -\frac{1}{2}\right]$.
\end{enumerate}
Therefore, 
\begin{equation*}
    u(x):=\left\{\begin{aligned}
     &\frac{1}{2}\left(|x|-\frac{1}{2}\right)^2 \quad \text{ in } \left[-1, -\frac{1}{2}\right]\cup \left[\frac{1}{2},1\right]\\
     &0 \qquad \qquad \qquad \text{ in } \left(-\frac{1}{2}, \frac{1}{2}\right).
    \end{aligned}
    \right.
\end{equation*}

\section*{Acknowledgement}
I would like to express thanks to Hung V. Tran for suggesting the problem, helpful conversations and invaluable advice.

\bibliographystyle{apalike}
\bibliography{ref}

\begin{thebibliography}{}

\bibitem[Albano, 2010]{Paolo2010}
Albano, P. (2010).
\newblock On the local semiconcavity of the solutions of the eikonal equation.
\newblock {\em Nonlinear Analysis: Theory, Methods and Applications},
  73:458--464.

\bibitem[Buckdahn et~al., 2010]{BRCQ2010}
Buckdahn, R., Cannarsa, P., and Quincampoix, M. (2010).
\newblock Lipschitz continuity and semiconcavity properties of the value
  function of a stochastic control problem.
\newblock {\em Nonlinear Differential Equations and Applications NoDEA},
  17:715--728.

\bibitem[Cannarsa and Castelpietra, 2008]{CANNARSA2008616}
Cannarsa, P. and Castelpietra, M. (2008).
\newblock Lipschitz continuity and local semiconcavity for exit time problems
  with state constraints.
\newblock {\em Journal of Differential Equations}, 245(3):616--636.

\bibitem[Cannarsa and Sinestrari, 2004]{cannarsa2004semiconcave}
Cannarsa, P. and Sinestrari, C. (2004).
\newblock {\em Semiconcave Functions, Hamilton-Jacobi Equations, and Optimal
  Control}.
\newblock Progress in Nonlinear Differen. Birkh{\"a}user Boston.

\bibitem[Capuzzo-Dolcetta and Lions, 1990]{Capuzzo-Dolcetta1990}
Capuzzo-Dolcetta, I. and Lions, P.-L. (1990).
\newblock {H}amilton--{J}acobi equations with state constraints.
\newblock {\em Transactions of the American Mathematical Society},
  318(2):643--683.

\bibitem[Caroff, 2006]{CAROFF2006287}
Caroff, N. (2006).
\newblock Semiconcavity of the value function for the bolza control problem.
\newblock {\em Journal of Mathematical Analysis and Applications},
  315(1):287--301.

\bibitem[Evans and Gomes, 2001]{evans2001effective}
Evans, L.~C. and Gomes, D. (2001).
\newblock Effective {H}amiltonians and averaging for {H}amiltonian dynamics
  {I}.
\newblock {\em Archive for rational mechanics and analysis}, 157(1):1--33.

\bibitem[Fathi, 2014]{Fathi2014}
Fathi, A. (2014).
\newblock Weak {KAM} {T}heorem in {L}agrangian {D}ynamics.
\newblock {\em Monograph}, 88.

\bibitem[Han and Tu, 2022]{YuTu2022}
Han, Y. and Tu, S. N.~T. (2022).
\newblock Remarks on the vanishing viscosity process of state-constraint
  {H}amilton-{J}acobi equations.
\newblock {\em Accepted to Applied Mathematics and Optimization}.

\bibitem[Ishii and Koike, 1996]{Ishii1996}
Ishii, H. and Koike, S. (1996).
\newblock A new formulation of state constraint problems for first-order
  {PDEs}.
\newblock {\em SIAM Journal on Control and Optimization}, 34(2):554--571.

\bibitem[Ishii and Loreti, 2002]{ishii_class_2002}
Ishii, H. and Loreti, P. (2002).
\newblock A {Class} of {Stochastic} {Optimal} {Control} {Problems} with {State}
  {Constraint}.
\newblock {\em Indiana University Mathematics Journal}, 51(5):1167--1196.

\bibitem[Ishii et~al., 2017]{ishii_vanishing_2017}
Ishii, H., Mitake, H., and Tran, H.~V. (2017).
\newblock The vanishing discount problem and viscosity {Mather} measures.
  {Part} 2: {Boundary} value problems.
\newblock {\em Journal de Mathématiques Pures et Appliquées}, 108(3):261 --
  305.

\bibitem[Kim et~al., 2020]{kim_state-constraint_2020}
Kim, Y., Tran, H.~V., and Tu, S.~N. (2020).
\newblock State-{constraint} {Static} {Hamilton}--{Jacobi} {Equations} in
  {Nested} {Domains}.
\newblock {\em SIAM Journal on Mathematical Analysis}, 52(5):4161--4184.

\bibitem[Lasry and Lions, 1989]{Lasry1989}
Lasry, J.~M. and Lions, P.-L. (1989).
\newblock Nonlinear elliptic equations with singular boundary conditions and
  stochastic control with state constraints. {I}. the model problem.
\newblock {\em Mathematische Annalen}, 283(4):583--630.

\bibitem[Mitake, 2008]{mitake_asymptotic_2008}
Mitake, H. (2008).
\newblock Asymptotic {Solutions} of {Hamilton}–{Jacobi} {Equations} with
  {State} {Constraints}.
\newblock {\em Applied Mathematics and Optimization}, 58(3):393--410.

\bibitem[Mou, 2016]{Chen2016}
Mou, C. (2016).
\newblock Semiconcavity of viscosity solutions for a class of degenerate
  elliptic integro-differential equations in {R}n.
\newblock {\em Indiana University Mathematics Journal}, 65(6):1891--1920.

\bibitem[Pignotti, 2005]{PIGNOTTI2005197}
Pignotti, C. (2005).
\newblock Semiconcavity results for constrained optimal control problems in a
  half-space.
\newblock {\em Journal of Mathematical Analysis and Applications},
  305(1):197--218.

\bibitem[Porretta, 2004]{por2004}
Porretta, A. (2004).
\newblock {Local estimates and large solutions for some elliptic equations with
  absorption}.
\newblock {\em Advances in Differential Equations}, 9(3-4):329--351.

\bibitem[{Porretta} and {Véron}, 2006]{alessio_asymptotic_2006}
{Porretta}, A. and {Véron}, L. (2006).
\newblock Asymptotic {Behaviour} of the {Gradient} of {Large} {Solutions} to
  {Some} {Nonlinear} {Elliptic} {Equations}.
\newblock {\em Advanced Nonlinear Studies}, 6(3):351--378.

\bibitem[Sinestrari, 1995]{Sinestrari1995SemiconcavityOS}
Sinestrari, C. (1995).
\newblock Semiconcavity of solutions of stationary {H}amilton-{J}acobi
  equations.
\newblock {\em Nonlinear Analysis-theory Methods and Applications},
  24:1321--1326.

\bibitem[Soner, 1986]{Soner1986}
Soner, H. (1986).
\newblock Optimal control with state-space constraint {I}.
\newblock {\em SIAM Journal on Control and Optimization}, 24(3):552--561.

\bibitem[Str\"omberg, 2010]{Thomas2010}
Str\"omberg, T. (2010).
\newblock Semiconcavity estimates for viscous {H}amilton–{J}acobi equations.
\newblock {\em Arch. Math.}, 94:579--589.

\bibitem[Tran, 2021]{tran_hamilton-jacobi_2021}
Tran, H.~V. (2021).
\newblock {\em Hamilton-{Jacobi} {Equations}: {Theory} and {Applications}}.
\newblock American Mathematical Society.

\bibitem[Tu, 2022]{tu2021vanishing}
Tu, S. N.~T. (2022).
\newblock Vanishing discount problem and the additive eigenvalues on changing
  domains.
\newblock {\em Accepted to Journal of Differential Equations}.

\end{thebibliography}
\end{document}